\documentclass[11pt,a4paper]{amsart}
\usepackage{amssymb}
\usepackage{amsmath}
\usepackage{amsfonts}
\usepackage[english]{babel}
\usepackage[T1]{fontenc}
\usepackage[latin1]{inputenc}
\usepackage{amsthm}
\usepackage[all]{xy}
\usepackage{comment}
\usepackage{tikz} 
\usepackage{color}
\usepackage{mathrsfs}

\newtheorem{theorem}{Theorem}[section]
\newtheorem{lemma}[theorem]{Lemma}

\newtheorem{proposition}[theorem]{Proposition}

\theoremstyle{definition}

\newtheorem{example}[theorem]{Example}

\newtheorem{prop}[theorem]{Proposition}
\theoremstyle{remark}
\newtheorem{remark}[theorem]{Remark}
\newtheorem{remarks}[theorem]{Remarks}
\numberwithin{equation}{section}

\theoremstyle{plain}

\title{Structure of closed subideals of $\mathcal L(X)$}

\begin{document}

\author[Hans-Olav Tylli and Henrik Wirzenius]{Hans-Olav Tylli and Henrik Wirzenius}

\thanks{H. Wirzenius is supported by the project L100192451 of the Czech Academy of Sciences.}

\address{Tylli: Department of Mathematics and Statistics, Box 68,
FI-00014 University of Helsinki, Finland}
\email{hans-olav.tylli@helsinki.fi}
\address{Wirzenius: Institute of Mathematics, Czech Academy of Sciences, \v{Z}itn\'a 25, 115 67 Praha 1, Czech Republic}
\address{Faculty of Information Technology and Communications Sciences, Tampere University,
P.O. Box 692, FI-33101, Tampere, Finland.}
\email{wirzenius@math.cas.cz}

\begin{abstract}
The closed subalgebra~$\mathcal J$ of the Banach algebra~$\mathcal L(X)$ of bounded linear operators on the Banach space~$X$ is a non-trivial closed 
~$\mathcal I$-subideal of~$\mathcal L(X)$ if~$\mathcal I$ is a closed ideal of~$\mathcal L(X)$ and~$\mathcal J$ is an ideal of~$\mathcal I$, but~$\mathcal J$ is not an ideal of~$\mathcal L(X)$.
We obtain a variety of examples of non-trivial closed subideals of~$\mathcal L(X)$ for different spaces~$X$,  which highlight further significant differences compared to the 
class of closed ideals.  We study the concept of a closed $n$-subideal of~$\mathcal L(X)$  for~$n \ge 3$, which is a natural generalization of that of a closed subideal.
In particular, we find explicit spaces~$X$ for which~$\mathcal L(X)$ contains a decreasing sequence~$(\mathcal M_n)_{n\in \mathbb N}$ of closed subalgebras, where for all $n\in\mathbb N$ the subalgebra $\mathcal M_n$ is an~$(n+1)$-subideal of $\mathcal L(X)$ but not an $n$-subideal. Moreover, we construct closed $n$-subideals contained in the compact operators $\mathcal K(X)$ for certain Banach spaces $X$ which fail the approximation property.
\end{abstract}

\maketitle

\section{Introduction}\label{intro}

Let $\mathcal A$ be a Banach algebra, and suppose that  
\[
\mathcal J \subset \mathcal I \subset \mathcal A
\]
are closed subalgebras. 
The subalgebra $\mathcal J$ is said to be a closed 
$\mathcal I$-\textit{subideal} of $\mathcal A$ if $\mathcal I$ is an ideal of $\mathcal A$ and $\mathcal J$ is an ideal of $\mathcal I$. The closed
$\mathcal I$-subideal~$\mathcal J$ is  \textit{non-trivial} if $\mathcal J$ is not an ideal of $\mathcal A$. 
For brevity we say that $\mathcal J$ is a non-trivial closed subideal of $\mathcal A$ if  $\mathcal J$ is a non-trivial closed  $\mathcal I$-subideal of $\mathcal A$  
for some closed ideal  $\mathcal I  \varsubsetneq  \mathcal A$.

Let $\theta:  \mathcal A \to  \mathcal B$ be a continuous  algebra isomorphism, where $\mathcal A$ and  $\mathcal B$ are Banach algebras. Then 
$\theta$ maps non-trivial closed subideals of $\mathcal A$ to non-trivial closed subideals of $\mathcal B$ (cf. Lemma \ref{qsubid} below), so the 
collection of closed subideals contains finer  information about the structure of the algebras compared with the closed ideals. 

Our main concern is the class of  Banach algebras $\mathcal L(X)$ consisting of the bounded operators on the (real or complex) Banach space $X$.
The first examples of non-trivial (and non-closed) subideals of $\mathcal L(\ell^2)$ were obtained by Fong and Radjavi  \cite{FR83}, though 
 the term subideal was coined later by Patnaik and Weiss \cite{PW13}. 
On the other hand,   $\mathcal L(X)$ fails to have non-trivial closed subideals if  $X$ is a Hilbert space, or if $X$ is one of  the spaces in
\begin{equation}\label{spaces}
\{\ell^p:  1 \le p < \infty\} \cup \{c_0\}.
\end{equation}
This is seen from the fact that the compact operators $\mathcal K(X)$ is the only proper closed ideal of $\mathcal L(X)$ for  $X$ in  \eqref{spaces},
see e.g. \cite[5.1-5.2]{Pi80}, 
and from \cite[Remark 3.2.(i)]{TW23} for arbitrary Hilbert spaces $X$.
At the other extreme, it was shown in \cite[Section 3]{TW23} that $\mathcal L(X)$  admits
large families of non-trivial  closed subideals  for many classical Banach spaces $X$, including $L^p(0,1)$ with $p \neq 2$.
 
In  Section \ref{Tarbard} we obtain examples of various intermediary phenomena for non-trivial closed subideals of $\mathcal L(X)$ by considering  
classes of spaces $X$, where either $\mathcal L(X)$ contains relatively few closed ideals or closed ideals of a specific type.
In particular, we discuss  closed subideals for  
the class of hereditarily indecomposable spaces $X_k$ constructed by Tarbard  \cite{Ta12a}, \cite{Ta12},
as well as the spaces $X$ obtained by Motakis in \cite{M23}, where the  Calkin algebra $\mathcal L(X)/\mathcal K(X)$ is isomorphic as a Banach algebra to 
$C(K)$ for compact metric spaces $K$. 

In \cite[Theorem 4.5]{TW22} we found a Banach space  $Z_p$ together with a family
\[
\mathfrak F = \{\mathcal I_B: \emptyset \neq B \varsubsetneq \mathbb N\}
\]
 the size of the continuum, which  consists of non-trivial closed $\mathcal K(Z_p)$-subideals of  $\mathcal L(Z_p)$.
The family $\mathfrak F$  has the rather unexpected property \cite[Theorem 2.1]{TW23} 
that $\mathcal I_{B_{1}}$ and $\mathcal I_{B_{2}}$ are isomorphic Banach algebras for any non-empty subsets $\emptyset \neq B_1, B_2 \varsubsetneq  \mathbb N$.
In Section \ref{newsubs}
we revisit this setting and construct another large family $\mathfrak F_+$ of non-trivial closed $\mathcal K(Z_p)$-subideals of  $\mathcal L(Z_p)$ such that $\mathfrak F \cap \mathfrak F_+ = \emptyset$.
The new subideals from $\mathfrak F_+$  are again isomorphic as Banach algebras to the subideals  $\mathcal I_B$ from $\mathfrak F$.

Shulman and Turovskii \cite{ST14} introduced the natural graded concept of an algebraic $n$-subideal for Banach algebras.
In Section \ref{nsubs}, which is the most significant part of this paper, we turn our attention to the closed $n$-subideals of $\mathcal L(X)$ for $n \ge 2$. Here the closed $2$-subideals correspond to the  subideals studied in  Sections 
 \ref{Tarbard}  and \ref{newsubs}. 
We construct  a variety of non-trivial closed $n$-subideals. Theorem \ref{nsubidchain} demonstrates that for certain direct sums such as $X = \big(\bigoplus_{j=1}^\infty \ell^{p_j}\big)_{c_0}$
the algebra $\mathcal L(X)$  contains an infinite decreasing sequence $(\mathcal M_n)$ of closed subalgebras, where for all $n\in\mathbb N$ the subalgebra $\mathcal M_n$ is an $(n+1)$-subideal
of $\mathcal L(X)$  but not an $n$-subideal.
Moreover, Theorem \ref{nsubid}  produces  for each $n\in\mathbb N$ a closed subspace $Z_n \subset c_0$ that fails to have the approximation property, for which $\mathcal L(Z_n)$ has a closed $(n+1)$-subideal 
$\mathcal M_{n} \subset \mathcal K(Z_n)$ that is not an $n$-subideal of $\mathcal L(Z_n)$.

The identification or even classification of the closed ideals of $\mathcal L(X)$ for various spaces $X$ is a topic of longstanding interest, see e.g. \cite[Chapter 5]{Pi80}. Recently there has been 
significant  advances, see e.g. \cite{SZ18}, \cite{JPS20}, \cite{JS21}, \cite{FSZ21} and their references. Closed $n$-subideals of $\mathcal L(X)$ for $n \ge 2$  concern a weaker ideal structure,
which exhibits many surprising properties.

We refer to \cite{LT77} and \cite{DJT95} for undefined concepts related to  Banach spaces and operator theory, and similarly to \cite{D00} and \cite{P94} for  Banach algebra theory. 
Throughout the paper $X \approx Y$ denotes linearly isomorphic Banach spaces $X$ and $Y$, and $\mathcal A \cong \mathcal B$ denotes (multiplicatively) isomorphic Banach algebras 
$\mathcal A$ and $\mathcal B$.

\section{Closed subideals associated to  H.I. type spaces}\label{Tarbard}

In this section we discuss  examples which highlight various differences between closed ideals and  non-trivial closed subideals of the algebras~$\mathcal L(X)$.
For instance, there are Banach spaces $X$ such that~$\mathcal L(X)$ has only finitely many closed ideals, but contains a continuum of non-trivial closed subideals (Theorem \ref{tarbard}). On the other hand, there are spaces $X$ 
for which~$\mathcal L(X)$ has a large class of closed ideals, but fails to contain any non-trivial closed subideals (Example \ref{mot}). 

These examples involve a precise understanding of 
the closed ideals of $\mathcal L(X)$, and some of the relevant spaces  will be drawn from the class of hereditarily indecomposable spaces. Recall that the Banach space $X$ is \textit{hereditarily indecomposable} (H.I.) 
if no closed infinite-dimensional subspace $M \subset X$ can be decomposed as a direct sum $M = M_1 \oplus M_2$, where $M_j \subset M$ are closed infinite-dimensional subspaces for $j = 1, 2$.
The first examples of H.I. spaces were constructed by Gowers and Maurey \cite{GM93}, and this class appears in many connections, see e.g. the survey \cite{Ma03}.

We state the following relevant  fact for explicit reference.

\begin{remark}\label{nonzero}
Let $X$ be an arbitrary Banach space and denote by $\mathcal F(X)$  the class of bounded finite-rank operators $X \to X$. 
If $\mathcal J$ is any non-zero closed subideal  of $\mathcal L(X)$, then 
$\mathcal A(X) \subset \mathcal J$, where $\mathcal A(X)  = \overline{\mathcal F(X)}$ is the closed ideal of
the approximable operators $X \to X$,  see e.g. \cite[Theorem 2.5.8.(ii)]{D00} or \cite[(2.1)]{TW23}.
In particular, if $X$ has the approximation property (A.P.), then
$\mathcal K(X) = \mathcal A(X)  \subset \mathcal J$ for any such $\mathcal J \neq \{0\}$.
\end{remark}

Argyros and Haydon \cite{AH11} constructed a real H.I. Banach space $X_{AH}$ having a  Schauder basis, such that  
\[
\mathcal L(X_{AH}) = \{\lambda I_{X_{AH}} + T: \lambda \in \mathbb R,\ T \in \mathcal K(X_{AH})\}.
\] 
The algebra $\mathcal L(X_{AH})$ does not have any non-trivial closed subideals, because $X_{AH}$ has the A.P. 
so that $\mathcal K(X_{AH})$ is the only non-trivial  closed ideal of $\mathcal L(X_{AH})$. 
Subsequently Tarbard \cite{Ta12a}, \cite{Ta12} 
obtained  H.I.  spaces $X_k$ for $k \ge 2$,  such that $\mathcal L(X_k)$ has relatively few operators and the 
closed ideals of $\mathcal L(X_k)$ are explicitly  classified (see \eqref{idchain} below). 

Let $X$ and $Y$ be Banach spaces. Recall that $T \in \mathcal L(X,Y)$ is a \textit{strictly singular} operator, denoted $T \in \mathcal S(X,Y)$, if the restriction of $T$ does not define a
linear isomorphism $M \to TM$ for any closed infinite-dimensional subspace $M \subset X$. The class $\mathcal S$ is a closed Banach operator ideal in the sense of 
Pietsch \cite{Pi80}, so that 
$\mathcal S(X)$ is a closed ideal of $\mathcal L(X)$  for any $X$, see e.g. \cite[Proposition 2.c.5]{LT77} or \cite[Section 1.9]{Pi80}.
Moreover, $\mathcal K(X,Y) \subset \mathcal S(X,Y)$ for any $X$ and $Y$.

It is a natural question whether the algebras $\mathcal L(X_k)$ admit any  non-trivial closed subideals. 
The answer  turns out to depend on $k$. 
Recall from  \cite{Ta12} that the Tarbard space $X_k$ is a separable real $\mathcal L^\infty$-space which has  a Schauder basis,  the dual space 
$X_k^* \approx \ell^1$, and that the following properties are satisfied for $k \ge 2$: 

(i) There is a strictly singular operator $S \in \mathcal S(X_k) \setminus \mathcal K(X_k)$ such that $S^{k-1} \notin \mathcal K(X_k)$ and $S^k = 0$, and 
every $U \in \mathcal L(X_k)$ has the unique representation
\begin{equation}\label{bded}
U = \alpha I_{X_{k}} + \sum_{j=1}^{k-1} \alpha_j S^j + T,
\end{equation}
where $\alpha, \alpha_1,\ldots, \alpha_{k-1} \in \mathbb R$ and $T \in \mathcal K(X_k)$. In particular, 
\[
\big(I_{X_{K}} + \mathcal K(X_k), S + \mathcal K(X_k), \ldots, S^{k-1} + \mathcal K(X_k)\big)
\]
 is a linear basis of the $k$-dimensional Calkin algebra $\mathcal L(X_k)/\mathcal K(X_k)$.

(ii) The non-zero closed ideals of $\mathcal L(X_k)$ form the chain 
\begin{equation}\label{idchain}
\mathcal K(X_k) \varsubsetneq [S^{k-1}] \varsubsetneq [S^{k-2}] \varsubsetneq \ldots  \varsubsetneq [S^{2}]
\varsubsetneq  [S]  = \mathcal S(X_k) \varsubsetneq \mathcal L(X_k).
\end{equation}
Here $[S^j]$ denotes the closed ideal of $\mathcal L(X_k)$ generated by the operator $S^j$ for $j = 1, \ldots, k-1$. Property   \eqref{bded} implies 
that
\begin{equation}\label{sj}
[S^j] = span(S^j,\ldots,S^{k-1}) + \mathcal K(X_k). 
\end{equation}
Hence  the quotients $\mathcal L(X_k)/ [S]$, $[S^{k-1}]/ \mathcal K(X_k)$, as well as $[S^{j}]/ [S^{j+1}]$ for $j = 1,\ldots, k-2$,
are $1$-dimensional.

We first point out  that in the above  setting
any closed $\mathcal S(X_k)$-subideal $\mathcal J$ is actually an ideal of $\mathcal L(X_k)$.

\begin{lemma}\label{codim}
Suppose that $X$ is a Banach space, and assume that the closed ideal $\mathcal I \subset \mathcal L(X)$ has codimension $1$ in $\mathcal L(X)$. 
Then every closed $\mathcal I$-subideal $\mathcal J$ is an ideal of $ \mathcal L(X)$.
\end{lemma}
\begin{proof}
Suppose  that  $V \in \mathcal J$ and $U = \alpha I_X + T \in \mathcal L(X)$ are arbitrary, where $T \in \mathcal I$. (Recall  that 
$I_X \notin \mathcal I$.)  It follows that
$VU =  \alpha V + VT \in \mathcal J$ and  $UV =  \alpha V + TV \in \mathcal J$, that is, $\mathcal J$  is an ideal of  $\mathcal L(X)$.
\end{proof}
For $n \ge 2$ there are spaces $X$ and closed ideals $\mathcal I \subset \mathcal L(X)$ having codimension $n$, such that $\mathcal L(X)$ 
contains non-trivial closed $\mathcal I$-subideals $\mathcal J$, see Remark \ref{ex1} below.

\begin{theorem}\label{tarbard}
(i)  The algebras $\mathcal L(X_2)$ and  $\mathcal L(X_3)$ do not have any non-trivial closed subideals.

(ii) Let $k \ge 4$ and fix $j \in \mathbb N$ such that $k/2 \le j \le k-2$. 
Let $S \in \mathcal S(X_k)$ be the operator that satisfies \eqref{bded} -- \eqref{sj}. 
Then the family $\mathfrak F_j$ consisting of the non-trivial closed $[S^j]$-subideals of $\mathcal L(X_k)$ equals  the class of closed subalgebras
\begin{equation}\label{full}
\mathcal I_N := N + \mathcal K(X_k), 
\end{equation}
where $\{0\}\neq N \subset span(S^j,\ldots,S^{k-1})$ is any linear subspace that satisfies
 \[
 N \neq span(S^r,\ldots,S^{k-1}), \textrm{ for }\  r = j,\ldots,k-1.
 \]
Each family $\mathfrak F_j$ has the size of the continuum.
\end{theorem}

\begin{proof}
(i) $\mathcal K(X_2)$ and $\mathcal S(X_2)$ are the only proper closed ideals $\mathcal I$ of $\mathcal L(X_2)$, and any non-zero closed
subideal $\mathcal J$ must satisfy  $\mathcal K(X_2) \subset \mathcal J \subset \mathcal S(X_2)$ in view of Remark \ref{nonzero}. Since $\mathcal S(X_2)/ \mathcal K(X_2)$
is $1$-dimensional, there are not even any  closed linear subspaces strictly between $\mathcal K(X_2)$ and $\mathcal S(X_2)$.

Analogously, if $\mathcal J$ is a non-trivial closed $\mathcal I$-subideal of $\mathcal L(X_3)$, then by \eqref{idchain} and  Lemma \ref{codim}
the intermediate closed ideal $\mathcal I = [S^2]$. However, once again there are no closed linear subspaces
$\mathcal K(X_3) \varsubsetneq \mathcal J \varsubsetneq [S^2]$,
since  the quotient  $[S^2]/ \mathcal K(X_3)$ is $1$-dimensional.

\smallskip

(ii) Note first that for any $k \ge 4$ it is possible to find $j \in \{2,\ldots, k-2\}$ such that $2j \ge k$. (In fact,  choose $j = 2$ for $k = 4$, $j = 3$ for $k = 5$,
and for $k \ge 6$  there are progressively more choices for $j$.)

We will require  the  cancellation property
\begin{equation}\label{prod}
UV \in \mathcal K(X_k) \textrm{ for } U \in  [S^j]  \textrm{ and } V \in [S^m], 
\end{equation}
whenever  $j$ and $m$ satisfy $j + m \ge k$. We will actually verify a more comprehensive statement, where the case $ j+m<k$ will  be relevant  in 
Theorem \ref{tarbard2}.

\smallskip

\textit{Claim}.  If $U\in [S^l]$ and $V\in [S^m]$, where $l, m  \in \{1, \ldots, k-1\}$, then
\begin{equation}\label{prod2}
UV \in [S^{l+m}] \  \text{ if } l+m<k, \quad UV \in  \mathcal K(X_k) \ \text{ if } l+m\ge k.
\end{equation}

In fact,  write $U  = \sum_{r=l}^{k-1} a_r S^r+ R_1$ and  $V  = \sum_{s=m}^{k-1} b_s S^s+ R_2$,
where  $R_1, R_2 \in \mathcal K(X_k)$.  If $l+m < k$  then we get  from $S^k = 0$ that 
\[
UV =  \sum_{u= 0}^{k-1-(l+m)} \Big(\sum_{t=0}^u a_{l+t}b_{m+u-t}\Big) S^{l+m+u} + R  \in  [S^{l+m}],
\]
as $R := \Big(\sum_{r=l}^{k-1} a_r S^r\Big)R_2 + R_1\Big(\sum_{s=m}^{k-1} b_s S^s\Big) + R_1R_2 \in \mathcal K(X_k)$.
Moreover, if $l+m \ge k$, then  $UV \in \mathcal K(X_k)$ since $(\sum_{r=l}^{k-1} a_r S^r)(\sum_{s=m}^{k-1} b_s S^s) = 0$
in view of the property that $S^k = 0$.

Next we make the following elementary observation: if $M$ and $N$ are distinct subspaces of $span(S^j,\ldots,S^{k-1})$, then 
\begin{equation}\label{0918}
M+\mathcal K(X_k)\neq N+\mathcal K(X_k).
\end{equation}
This follows from linear algebra and the fact that \[span(S^j,\ldots,S^{k-1}) \cap \mathcal K(X_k) = \{0\}.\] 

We proceed by showing that $\mathcal I_N:=N+\mathcal K(X_k)$ is a non-trivial closed $[S^j]$-subideal of $\mathcal L(X_k)$ for any non-zero linear subspace 
$N \subset span(S^j,\ldots,S^{k-1})$ such that $N \neq span(S^r,\ldots,S^{k-1})$ for $r = j, \ldots, k-1$. By definition  $\mathcal I_N$ is a closed linear subspace of $[S^j]$, and \eqref{prod} together with the assumption $2j \ge k$ gives that
\[
UA,\ AU  \in \mathcal K(X_k) \subset \mathcal I_N
\]
for any $U \in \mathcal I_N$ and $A \in [S^j]$. Thus $\mathcal I_N$ is a closed ideal of $[S^j]$. On the other hand, $\mathcal I_N$ is not a closed ideal of $\mathcal L(X_k)$ in view of \eqref{idchain}, \eqref{sj} and \eqref{0918}.

Next, suppose that $\mathcal M$ is a non-trivial closed $[S^j]$-subideal of  $\mathcal L(X_k)$. By Remark \ref{nonzero} and \eqref{sj} we have
\[\mathcal K(X_k)\subsetneq\mathcal M \subsetneq span(S^j,\ldots,S^{k-1})+\mathcal K(X_k)\]
and thus $\mathcal M=N+\mathcal K(X_k)$ for some linear subspace $N$ of $span(S^j,\ldots,S^{k-1})$. Moreover, $N \neq span(S^r,\ldots,S^{k-1})$  for  any $r = j,\ldots,k-1$ by \eqref{idchain} and \eqref{sj} since $\mathcal M$ is not a closed ideal of $\mathcal L(X_k)$.

Finally, since $\#(\{j,\ldots,k-1\})\ge 2$ by assumption, there is  a continuum of distinct subspaces $N \subset span(S^j,\ldots,S^{k-1})$ for which $N \neq span(S^r,\ldots,S^{k-1})$ for $r = j, \ldots, k-1$. It follows from \eqref{0918} that there is a continuum of non-trivial closed $[S^j]$-subideals of $\mathcal L(X_k)$.

Conclude that  the family $\mathfrak F_j$ in \eqref{full} has the size of the continuum
for each $j \in \mathbb N$ that satisfies  $k/2 \le j \le k-2$.
\end{proof}

\begin{remark}\label{ex1}
Let  $r\ge 2$. By Theorem \ref{tarbard} there are (continuum many)  non-trivial closed $[S^r]$-subideals of $\mathcal L(X_{2r})$, where  $dim(\mathcal L(X_{2r})/ [S^r]) = r$ in view of \eqref{bded} and \eqref{sj}. Hence Lemma \ref{codim}  does not have  a general counterpart  for the closed $\mathcal I$-subideals of $\mathcal L(X)$, 
where $\mathcal I \subset \mathcal L(X)$ has codimension $r \ge 2$.
\end{remark}

The following simple fact provides a systematic tool for producing non-trivial closed (sub)ideals of certain algebras $\mathcal L(X)$. 
Special instances have been used  e.g. in \cite[Proposition 2.1]{TW22} and \cite[Proposition 5.2]{W23}. 

\begin{lemma}\label{qsubid}
Let $X$ be a Banach space and suppose that  $q: \mathcal L(X) \to \mathcal A$ is a continuous surjective algebra homomorphism onto the 
 Banach algebra $\mathcal A$.

(i) If $\mathcal J$ is a non-trivial closed $\mathcal I$-subideal of $\mathcal L(X)$  such that $ker(q) \subset \mathcal J$, then $q(\mathcal J)$ is a non-trivial closed 
$q(\mathcal I)$-subideal of $\mathcal A$.

(ii) If $\mathcal J$ is a non-trivial closed $\mathcal I$-subideal of $\mathcal A$, then 
$q^{-1}(\mathcal J)$ is a non-trivial closed $q^{-1}(\mathcal I)$-subideal of $\mathcal L(X)$ for which $ker(q) \subset q^{-1}(\mathcal J)$.

Hence there is a bijective  correspondence between the non-trivial closed subideals $ker(q) \subset \mathcal J \subset \mathcal L(X)$ and  the non-trivial closed subideals of 
$\mathcal A$.
\end{lemma}

\begin{proof}
(i) We first verify that $q(\mathcal J)$ is closed in $\mathcal A$. Suppose  that $q(S) \in \overline{q(\mathcal J)}$ is arbitrary,
and let $(S_n) \subset \mathcal J$ be a sequence for which $q(S_n - S) \to 0$ as $n \to \infty$. Since $q$ induces a linear isomorphism
$\mathcal L(X)/ker(q) \to \mathcal A$, there is a sequence $(R_n) \subset ker(q) \subset \mathcal J$ for which $||S - S_n - R_n|| \to 0$  in $\mathcal L(X)$ as $n \to \infty$. 
Thus $S\in\mathcal J$ since   $\mathcal J$ is a closed subalgebra and $S_n + R_n \in \mathcal J$ for all $n \in \mathbb N$. Similarly $q(\mathcal I)$ is closed in $\mathcal A$.

Since  $q$ is a Banach algebra homomorphism  we get that $q(\mathcal J) \subset q(\mathcal I)$ and $q(\mathcal I) \subset \mathcal A$ are respective ideals.
Finally, $q(\mathcal J)$ is not an ideal of $\mathcal A$. In fact, suppose that $S \in \mathcal J$ and $U \in \mathcal L(X)$ satisfy $SU \notin \mathcal J$. It follows that
$q(S)q(U) = q(SU) \notin q(\mathcal J)$, since otherwise $SU - R \in ker(q)$ for some $R \in \mathcal J$ and hence $SU \in \mathcal J$. 
The case  where $US \notin \mathcal J$ is similar. Conclude that $q(\mathcal J)$ is a non-trivial closed 
$q(\mathcal I)$-subideal of $\mathcal A$.

(ii) The argument is an easier variant of part (i), and the details are left to the readers.
\end{proof}

There are interesting  variations of Theorem  \ref{tarbard}  for   certain spaces  introduced by Kania and Laustsen \cite{KK17}.
Let $m_1,\ldots,m_n\in\mathbb N$ be given. They constructed in \cite[Note added in proof]{KK17} a
space $Z = Y_1^{m_1}\oplus\cdots\oplus Y_n^{m_n}$, for which the Calkin algebra  
\[
\mathcal L(Z)/\mathcal K(Z)\cong M_{m_1}(\mathbb K)\oplus\cdots\oplus M_{m_n}(\mathbb K).
\] 
Here $M_{r}(\mathbb K)$ is the algebra of scalar $r \times r$-matrices. The spaces $Y_j$ were constructed in \cite[Theorem 10.4]{AH11}, and they satisfy 
$\mathcal L(Y_j,Y_k) = \mathcal K(Y_j,Y_k)$ for  $j \neq k$ with $j, k \in \{1,\ldots, n\}$. In particular,
\begin{equation}\label{KLnil2} 
\mathcal L(Y_j^{m_j}, Y_k^{m_k}) = \mathcal K(Y_j^{m_j}, Y_k^{m_k}) \textrm{ for }  j \neq k,
\end{equation} 
see (i) below. It is pointed out in  \cite[page 1023]{KK17} that the non-zero
closed ideals of $\mathcal L(Z)$ are precisely
\begin{equation}\label{klids2}
\mathcal I_N :=\{T= [T_{jk}]_{j,k = 1}^n : T_{jj} \in \mathcal K(Y_j^{m_j}) \text{ for all }j\notin N\},
\end{equation}
where $N\subset\{1,\ldots,n\}$ is any subset.

We recall that the operator $T\in\mathcal L(X,Y)$ between finite direct sums $X=X_1\oplus\cdots\oplus X_n$ and $Y=Y_1\oplus\cdots\oplus Y_m$ can be represented as an $m\times n$-operator matrix $T=[T_{jk}]$, where $T_{jk} = P_{Y_j}TJ_{X_k} \in \mathcal L(X_k,Y_j)$ and $P_{Y_j}: Y \to Y_j$,  respectively $J_{X_k}: X_k\to X$, are the natural projections and inclusions related to
the components of  the direct sums $X$ and $Y$. The following facts are well known and easy to check by using the identities $I_X=\sum_{r=1}^n J_{X_r}P_{X_r}$ and $I_Y=\sum_{r=1}^m J_{Y_r}P_{Y_r}$.
\smallskip

(i) If $T\in\mathcal L(X,Y)$, then \[T=\sum_{j=1}^m\sum_{k=1}^n J_{Y_j}T_{jk}P_{X_k}.\] In particular,
  $T\in\mathcal K(X,Y)$ if and only if $T_{jk}\in\mathcal K(X_k,Y_j)$ for all  $1\le k\le n$ and $1\le j\le m$.
\smallskip
  
  (ii) If $T,U\in\mathcal L(X)$, then \begin{equation}\label{productofoperators}[TU]_{jk}=\sum_{r=1}^n T_{jr}U_{rk}\end{equation} for all $1\le j,k\le n$.
  
\begin{example}\label{KL1}
Let $Z = Y_1^{m_1}\oplus\cdots\oplus Y_n^{m_n}$, where $m_1,\ldots,m_n\in\mathbb N$.
Then the Banach algebra $\mathcal L(Z)$ does not have any non-trivial closed subideals.
\end{example} 

\begin{proof}
Fix  $N\subset\{1,\ldots,n\}$ and let $\mathcal I_N$ be the closed ideal of $\mathcal L(Z)$ from \eqref{klids2}. Let $q: \mathcal L(Z) \to \mathcal L(Z)/\mathcal K(Z)$ be the quotient homomorphism. Since $\ker q=\mathcal K(Z)\subset \mathcal I_N$ by \eqref{klids2}, it suffices in view of Lemma \ref{qsubid} to verify that there are no non-trivial closed $q(\mathcal I_N)$-subideals of $\mathcal A:=\mathcal L(Z)/\mathcal K(Z)$. 

Define $U=[U_{jk}]_{j,k=1}^n\in\mathcal L(Z)$  by 
\[
U_{jk}=\begin{cases}
I_j,& \text{ if } j=k\in N\\
0,&\text{ otherwise}, 
\end{cases}
\]
where $I_j$ is the identity operator on the component $Y_j^{m_j} \subset Z$. Note that  $U\in \mathcal I_N$. We claim that $q(U)$ is the unit element of $q(\mathcal I_N)$.
 For this, let $T\in\mathcal I_N$. We observe that
\[[T-TU]_{jj}=T_{jj}-[TU]_{jj}=T_{jj}-\sum_{k=1}^n T_{jk}U_{kj}=\begin{cases}
0,& \text{if } j\in N\\
T_{jj}, & \text{if } j\notin N 
\end{cases}
\]
for all $1\le j\le n$. It follows that $T-TU\in\mathcal K(Z)$ in view of \eqref{KLnil2} and \eqref{klids2}. In a similarly way one verifies that $UT-T\in\mathcal K(Z)$. Consequently, $q(U)q(T)=q(UT)=q(T)$ and $q(T)q(U)=q(T)$. 

This implies that there are no non-trivial closed $q(\mathcal I_N)$-subideals of $\mathcal A$. In fact, assume that $\mathcal J$ is a closed $q(\mathcal I_N)$-subideal of $\mathcal A$. If $v\in\mathcal J$ and $a\in\mathcal A$, then 
\[va=(vq(U))a=v(q(U)a)\in \mathcal J\]
and similarly $av\in\mathcal J$.
\end{proof}

Next we provide  examples of Banach spaces $X$, where $\mathcal L(X)$ has an infinite family of closed ideals but fails to have  any non-trivial closed subideals.  For this purpose we will use the spaces recently constructed by 
Motakis \cite[Theorem A]{M23}: \textit{for any compact metric space $K$ there is 
a separable $\mathcal L^\infty$-space $X_{C(K)}$, such that   $X^*_{C(K)} \approx \ell^1$ and}
 \begin{equation}\label{Calkin1}
\mathcal L(X_{C(K)})/\mathcal K(X_{C(K)}) \cong C(K)
 \end{equation}
\textit{are isomorphic as Banach algebras}. Here $C(K)$ is the sup-normed Banach algebra of continuous scalar-valued functions defined on  $K$.
The construction of $X_{C(K)}$ is highly complicated, but we will only require the existence of a surjective Banach algebra homomorphism  $\mathcal L(X_{C(K)}) \to C(K)$ 
as given by \eqref{Calkin1}.   The results in  \cite{M23} apply to both  real and complex scalars.

Recall that if $K$ is a compact  metric space (or even a compact Hausdorff topological space), then 
the family of closed ideals $\mathcal I \subset C(K)$ coincides with  
\begin{equation}\label{cont}
\{\mathcal I(L): L \subset K \textrm{ is a closed subset}\} ,
\end{equation}
where $\mathcal I(L) = \{f \in C(K): f_{|L} = 0\}$.
We refer  e.g. to \cite[Theorem 1.4.6]{K09} for an argument that also applies to real $C(K)$-algebras.

\begin{example}\label{mot}
Let $K$ be any compact metric space and  let $X_{C(K)}$  be the space from \cite{M23}  for which  \eqref{Calkin1} holds. Fix a surjective  
Banach algebra homomorphism  $q_K: \mathcal L(X_{C(K)}) \to C(K)$ such that  $ker(q_K) = \mathcal K(X_{C(K)})$.

Then the closed ideals of $\mathcal L(X_{C(K)})$ coincide with the family
\begin{equation}\label{liftid}
\{q_K^{-1}(\mathcal I(L)): L \subset K \textrm{ is a closed subset}\},
\end{equation}
but $\mathcal L(X_{C(K)})$  does not have any non-trivial closed subideals. 
\end{example}

\begin{proof}
According to the construction in  \cite{M23} the spaces $X_{C(K)}$ have a Schauder basis,  so $ker(q_K) = \mathcal K(X_{C(K)})$ is the smallest  non-zero closed (sub)ideal
of $\mathcal L(X_{C(K)})$ by Remark \ref{nonzero}.
Hence  \eqref{cont} together with a similar argument as in Lemma \ref{qsubid}  imply that the closed ideals of $\mathcal L(X_{C(K)})$ are given by \eqref{liftid}.

Lemma  \ref{qsubid} implies that if  $\mathcal J$ is a non-trivial closed  $\mathcal I$-subideal of $\mathcal L(X_{C(K)})$, then 
$q_K(\mathcal J)$ is a non-trivial  closed $q_K(\mathcal I)$-subideal of $C(K)$. 
On the other hand,  the Banach algebras $C(K)$  do not have any non-trivial closed subideals. In fact,  the elementary 
metric version of Urysohn's lemma implies  that
any closed ideal $\mathcal I(L)$ of $C(K)$ has a bounded approximate identity, so that
any closed $\mathcal I(L)$-subideal $\mathcal J$ is  an ideal of $C(K)$ by \cite[Lemma 3.1]{TW23}.
(Alternatively, for complex scalars one may appeal to  e.g. \cite[Theorem 3.2.21]{D00}.) 
\end{proof}

\begin{remarks}\label{MPZ}
 (1)  The family of Banach spaces 
 \[
 \{X_{C(K)}: K \textrm{ is an infinite compact metric space}\}
 \]
  is fairly large:  if $K_1$ and $K_2$ are not homeomorphic metric spaces, then
 $X_{C(K_1)}$ and $X_{C(K_2)}$ are not  linearly isomorphic spaces. 
Namely, if  $U: X_{C(K_2)} \to X_{C(K_1)}$ is a linear isomorphism, then  $S \mapsto \psi(S) := U^{-1}SU$ 
 defines  a Banach algebra isomorphism $\mathcal L(X_{C(K_1)}) \to \mathcal L(X_{C(K_2)})$ for which $\psi(\mathcal K(X_{C(K_1)}) = \mathcal K(X_{C(K_2)})$.
Thus  $\psi$ induces a Banach algebra isomorphism
\[
\mathcal L(X_{C(K_1)})/\mathcal K(X_{C(K_1)}) \to  \mathcal L(X_{C(K_2)})/\mathcal K(X_{C(K_2)})
\]
through $S + \mathcal K(X_{C(K_1)}) \mapsto \psi(S) + \mathcal K(X_{C(K_2)})$, so that  $C(K_1) \cong C(K_2)$.
Moreover, if  $K_1$ and $K_2$ are compact metric spaces, then 
$C(K_1) \cong C(K_2)$ if and only if 
$K_1$ and $K_2$ are homeomorphic spaces. (See  e.g.  \cite[Theorem IV.6.26 and Corollary IV.6.27]{DS57}, which also applies to real scalars.)
\smallskip

(2) For each countable compact metric space $K$ the earlier result \cite[Theorem 5.1]{MPZ16} constructs a real Banach space $X_K$ of similar type as $X_{C(K)}$ 
for which \eqref{Calkin1} holds. The analogue of Example \ref{mot} also holds here. 
\end{remarks}

More restricted conclusions are available for earlier constructions of strange quotient algebras of particular $\mathcal L(X)$-algebras.

\begin{remarks}\label{MLW}
(1) Mankiewicz \cite[Theorem 1.1]{M89} constructed a real reflexive Banach space $X_M$ having a finite-dimensional decomposition, such that 
there is a surjective   Banach algebra homomorphism 
$q: \mathcal L(X_M) \to C(\beta \mathbb N)$, 
where  $\beta \mathbb N$ is the Stone-\v{C}ech compactification of $\mathbb N$. By arguing as in Example \ref{mot}  it follows that 
there does not exist any non-trivial closed subideals $\mathcal J$ of $\mathcal L(X_M)$ that satisfy 
\begin{equation}\label{mank}
ker(q) \varsubsetneq \mathcal J \varsubsetneq \mathcal L(X_M).
\end{equation}
The construction of $X_M$ involves random methods, and we are not aware of an explicit description of the ideal $ker(q) \subset \mathcal L(X_M)$.

\smallskip

(2)  Let   $J_p$ be the $p$-James space defined in terms of the $p$-variation norm for $p \in (1,\infty)$, see e.g. \cite[page 344]{LW89}
or \eqref{jp} from Section \ref{nsubs}. Let 
\[
Y = \Big(\bigoplus_{k \in \mathbb Z} J_{p_{k}}\Big)_{\ell^1}\ ,
\]
where  $(p_k)_{k\in \mathbb Z} \subset (1,\infty)$ is a strictly increasing sequence.
 Dales, Loy and Willis \cite[Proposition 3.5]{DLW94} found an explicit surjective Banach algebra homomorphism $q: \mathcal L(Y) \to \ell^\infty(\mathbb  Z)\cong C(\beta \mathbb Z)$, where
 $\mathcal K(Y) \varsubsetneq \mathcal W(Y) \subset ker(q)$. Here $\mathcal W(Y)$  is the closed ideal of $\mathcal L(Y)$ consisting of the weakly compact operators on $Y$. Thus the analogue of \eqref{mank} also holds in $\mathcal L(Y)$.
 \end{remarks}


\section{New examples of non-trivial closed $\mathcal K(Z_p)$-subideals of $\mathcal L(Z_p)$}\label{newsubs}

A family $\mathfrak F = \{\mathcal I_B: \emptyset \neq B \varsubsetneq \mathbb N\}$ of non-trivial closed $\mathcal K(Z_p)$-subideals was
 constructed in  \cite[Theorem 4.5]{TW22} for certain direct sums $Z_p$.
It was subsequently shown in \cite[Theorem 2.1]{TW23} that under a natural  condition on $Z_p$ the family $\mathfrak F$ has the property that 
 $\mathcal I_{B_{1}}$ and $\mathcal I_{B_{2}}$ are isomorphic  Banach algebras for any subsets  $\emptyset \neq B_1, B_2 \varsubsetneq \mathbb N$. It is known, 
see e.g. \cite[Added in Proof]{JS21}, that such a behaviour is not possible among the closed ideals of $\mathcal L(X)$ for any Banach space $X$. 
These results motivate  a more careful analysis of the structure of closed subideals of $\mathcal L(Z_p)$, where many interesting phenomena 
can be seen explicitly. In this section we find a new concrete family   
of  non-trivial closed $\mathcal K(Z_p)$-subideals of $\mathcal L(Z_p)$, which is disjoint from $\mathfrak F$ and which has  the size of the continuum. Recall also that non-trivial closed $\mathcal K(X)$-subideals can only exist within the class of Banach spaces that fail the approximation property.

\smallskip

We start by  recalling the setting of  \cite[Theorem 4.5]{TW22}.  Let $X$ be a Banach space such that $X$ has the A.P., but $X^*$ fails the A.P.  
It follows from  known results, see e.g. \cite[Theorem 1.e.5]{LT77}, that  there is a Banach space $Y$ such that
$\mathcal A(X,Y) \varsubsetneq \mathcal K(X,Y)$. Define
\begin{equation}\label{dir}
Z_p = \Big( \bigoplus_{j=0}^\infty X_j\Big)_{\ell^p},
\end{equation}
where $X_0 = Y$ and $X_j = X$ for $j \ge 1$.  Here $1 < p < \infty$ is fixed.
Let $\emptyset \neq B \varsubsetneq \mathbb N$ and put
\begin{equation}\label{osubid}
\mathcal I_B = \{S = [S_{jk}] \in \mathcal K(Z_p): S_{00} \in \mathcal A(Y),\ S_{0k} \in \mathcal A(X,Y)  \textrm{ for } k \in B\}.
\end{equation}
Here $S_{jk} = P_jSJ_k \in \mathcal K(X_k,X_j)$ for $j, k \in \mathbb N_0=\mathbb N\cup\{0\}$ by the ideal property of $\mathcal K$, 
where $P_j: Z_p \to X_j$ and $J_k: X_k \to Z_p$ are the natural projection and inclusion operators 
associated to the component spaces of the direct sum $Z_p$. We will occasionally use the alternative notation $S_{jk} = [S]_{jk}$ if this is  more convenient.
 
It was shown in  \cite[Theorem 4.5]{TW22}  that the family 
\begin{equation}\label{fam}
\mathfrak F = \{\mathcal I_B: \emptyset \neq B \varsubsetneq \mathbb N\} 
\end{equation}
consists of non-trivial closed $\mathcal K(Z_p)$-subideals of $\mathcal L(Z_p)$, which has the reversed order structure compared to $\mathcal P(\mathbb N)$.
Moreover, if $X$ satisfies the linear isomorphism condition
\begin{equation}\label{iso}
\Big( \bigoplus_{j\in \mathbb N} X\Big)_{\ell^p} \approx X,
\end{equation}
then $\mathcal I_{B_{1}} \cong \mathcal  I_{B_{2}}$
are isomorphic  Banach algebras for any $B_1 \neq B_2$, see \cite[Theorem 2.1]{TW23}. 

For any  non-zero sequence $c = (c_j) \in \ell^1$, let  
\[
supp(c) = \{j \in \mathbb N: c_j \neq 0\}
\]
be  the support of $c$.  We introduce 
\begin{align}\label{def0}
\mathcal I(c) := & \big\{S = [S_{jk}] \in \mathcal K(Z_p):  S_{00} \in \mathcal A(Y) \textrm{ and }\\
& x \mapsto \sum_{r=1}^\infty c_rS_{0r}x \textrm{ is an approximable operator } X \to Y\big\}\notag.
\end{align}

The series  $\sum_{r=1}^\infty c_rS_{0r}$ is absolutely 
convergent in the operator norm of $\mathcal K(X,Y)$, since  $c = (c_j) \in \ell^1$. 
The support $supp(c)$ can be any finite or infinite set and we will assume that the cardinality $\#(supp(c)) \ge 2$. 
Namely, if $supp(c) = \{k\}$, then we see by comparing \eqref{def0} and \eqref{osubid} that $\mathcal I(c) = \mathcal I_{\{k\}} \in \mathfrak F$.

\smallskip

The following theorem contains  the first main result of Section \ref{newsubs}. We will repeatedly use the facts that 
$\mathcal K(Y,X) = \mathcal A(Y,X)$ and $\mathcal K(X) = \mathcal A(X)$ as $X$ has the A.P.

\begin{theorem}\label{newsubids}
Let $Z_p$ be the space  from \eqref{dir}, and assume that the non-zero sequence $c = (c_j)  \in \ell^1$  satisfies  $\#(supp(c)) \ge 2$. Then the following properties hold:
\begin{enumerate}
\item[(i)]   $\mathcal I(c)$ is a closed ideal of $\mathcal K(Z_p)$, as well as a left ideal of $\mathcal L(Z_p)$.

\item[(ii)]   $\mathcal I(c)$ is not a right ideal of $\mathcal L(Z_p)$, so that $\mathcal I(c)$ is a non-trivial closed $\mathcal K(Z_p)$-subideal of $\mathcal L(Z_p)$.

\item[(iii)] $\mathcal I(c) \neq \mathcal I_B$ for any subset  $\emptyset \neq B \varsubsetneq \mathbb N$.

\item[(iv)]   If $d \in  \ell^1$ is a non-zero sequence such that $supp(d) \neq supp(c)$, then $\mathcal I(d) \neq \mathcal I(c)$.
 
\item[(v)]  Let $d\in\ell^1$ be a non-zero sequence. Then
$\mathcal I(c)=\mathcal I(d)$ if and only if $c=\lambda d$ for some non-zero scalar $\lambda$.
\end{enumerate}

Hence the family
\[
\mathfrak F_+ = \{\mathcal I(c): c \in  \ell^1 \textrm{ and }  \#(supp(c)) \ge 2\} 
\]
consists of non-trivial closed $\mathcal K(Z_p)$-subideals  of $\mathcal L(Z_p)$. The family $\mathfrak F_+$  has the size of the continuum,  and $\mathfrak F_+  \cap \mathfrak F = \emptyset$.
\end{theorem}
 
\begin{proof}
(i) Note first that  for any $c = (c_j) \in \ell^1$ the map 
\[
\psi_c: S \mapsto \sum_{j=1}^\infty c_jS_{0j}, \quad \mathcal L(Z_p) \to \mathcal L(X,Y), 
\]
defines a bounded linear operator, since for $x \in X$ and $S \in \mathcal L(Z_p)$ we have 
\[
\Vert  \sum_{j=1}^\infty c_jS_{0j}x\Vert \le  \sum_{j=1}^{\infty} \vert c_j\vert \Vert S_{0j}\Vert \Vert x\Vert \le    \Vert c \Vert_1 \Vert S\Vert \Vert x \Vert,
\]
as $\Vert S_{0j}\Vert = \Vert P_0SJ_j\Vert \le \Vert S\Vert$ for $j \in \mathbb N$.
Deduce that 
\[
 \mathcal I(c) = \psi_c^{-1}(\mathcal A(X,Y)) \cap \{S \in \mathcal K(Z_p): S_{00} \in \mathcal A(Y)\}
\]
is a closed linear subspace of $\mathcal K(Z_p)$. 

We next show that $\mathcal I(c)$ is a left ideal of $\mathcal L(Z_p)$.
Let $S = [S_{jk}] \in \mathcal I(c)$ and $U = [U_{jk}] \in \mathcal L(Z_p)$ be arbitrary.  
According to \eqref{kprod} from the subsequent Lemma \ref{tech1},  for each $j \in \mathbb N$ 
\begin{equation}\label{lprod}
[US]_{0j} = \sum_{l = 0}^\infty U_{0l}S_{lj} = U_{00}S_{0j} + \sum_{l = 1}^\infty U_{0l}S_{lj} =  U_{00}S_{0j}  + R_j 
\end{equation}
is a norm convergent series in $\mathcal K(X,Y)$, where we denote the operator norm limit by
\[
R_j := \lim_{r\to\infty} \sum_{l = 1}^r U_{0l}S_{lj} \in \mathcal K(X,Y).
\]
Actually $R_j  \in \mathcal A(X,Y)$ for $j \in \mathbb N$, since $S_{lj} \in \mathcal A(X)$ for all $l \ge 1$.
From \eqref{lprod} we get  for any fixed $r \in \mathbb N$ that
\begin{equation}\label{fin}
x \mapsto \sum_{j=1}^r c_j[US]_{0j}x =  U_{00}\big(\sum_{j=1}^r  c_jS_{0j}x\big) + \sum_{j=1}^r c_jR_jx, \quad x \in X.
\end{equation}
Since $(c_j) \in \ell^1$ we let $r \to \infty$ in \eqref{fin}, and deduce that the pointwise operator
\[
x \mapsto \sum_{j=1}^\infty c_j[US]_{0j}x =  U_{00}\big(\sum_{j=1}^\infty  c_jS_{0j}x\big) + \sum_{j=1}^\infty c_jR_jx
\]
is approximable $X \to Y$.
Here  we use  the facts that $x \mapsto \sum_{j=1}^\infty  c_jS_{0j}x$ defines an approximable operator $X \to Y$
by assumption, and that $R_j \in \mathcal A(X,Y)$ for $j \in \mathbb N$.
Moreover, Lemma \ref{tech1} below implies that
\[
[US]_{00} = U_{00}S_{00} + \sum_{l = 1}^\infty U_{0l}S_{l0} \in \mathcal A(Y),
\]
as $S_{00}  \in \mathcal A(Y)$ and $S_{l0}  \in \mathcal A(Y,X)$ for $l \ge 1$. We conclude that  $US \in \mathcal I(c)$.

\smallskip

There remains to verify that $\mathcal I(c)$ is a right ideal of  $\mathcal K(Z_p)$. Towards this suppose that $S = [S_{jk}] \in \mathcal I(c)$ and $U = [U_{jk}] \in \mathcal K(Z_p)$. 
In order to show that $SU \in \mathcal I(c)$ recall first that $U_{lj} \in \mathcal K(X) = \mathcal A(X)$ for any $j,l \ge 1$.

This implies that for each fixed $j \in \mathbb N$ the series 
\begin{equation}\label{rprod}
[SU]_{0j} = \sum_{l = 0}^\infty S_{0l}U_{lj} = S_{00}U_{0j} + \sum_{l = 1}^\infty S_{0l}U_{lj} =  S_{00}U_{0j}  + T_j
\end{equation}
is norm convergent in $\mathcal K(X,Y)$ by \eqref{kprod} in Lemma \ref{tech1} below, where as above 
\[
T_j := \lim_{s\to\infty}  \sum_{l = 1}^sS_{0l}U_{lj} \in \mathcal A(X,Y).
\]
In particular, $[SU]_{0j} \in  \mathcal A(X,Y)$ for $j \in \mathbb N$ as $S_{00} \in \mathcal A(Y)$ by assumption. Thus the operator
$x \mapsto \sum_{j=1}^\infty c_j[SU]_{0j}x$ is approximable $X \to Y$  because  $(c_j) \in \ell^1$.
Moreover, 
\[
[SU]_{00} = S_{00}U_{00} + \sum_{l = 1}^\infty S_{0l}U_{l0} \in \mathcal A(Y),
\]
where again $S_{00}  \in \mathcal A(Y)$ and $U_{l0} \in \mathcal A(Y,X)$ for  $l \ge 1$.
Conclude that $SU \in \mathcal I(c)$. In particular, $\mathcal I(c)$
is a closed ideal of $\mathcal K(Z_p)$.

\smallskip

For the proofs of parts (ii)-(v) we fix an operator
\begin{equation}\label{fix}
S_0 \in \mathcal K(X,Y) \setminus \mathcal A(X,Y).
\end{equation}
This is possible since  $\mathcal A(X,Y) \varsubsetneq \mathcal K(X,Y)$ in the construction of $Z_p$. We also require the following observation: let $r,s\in\mathbb N$ where $s\in supp(c)$. Then the operator
\begin{equation}\label{091725}
T_{(r,s)}:=J_0S_0P_r - (c_r/c_s) J_0S_0P_s\in\mathcal I(c).
\end{equation}
 
In fact, $T_{(r,s)}\in\mathcal K(Z_p)$ by \eqref{fix} and $[T_{(r,s)}]_{00} = 0$. Moreover, for any $x \in X$ we have
\begin{align*}
\sum_{j=1}^\infty c_j[T_{(r,s)}]_{0j}x &=\sum_{j=1}^\infty c_jP_0\big(J_0S_0P_r - (c_r/c_s) J_0S_0P_s\big)J_jx\\&= c_rS_0x - c_s(c_r/c_s)S_0x  = 0.
\end{align*}

(ii) Let $r,s\in supp(c)$, $r\neq s$, and consider the operator $T:=T_{(r,s)}\in\mathcal I(c)$ defined in \eqref{091725}. Let $U = J_rP_r \in \mathcal L(Z_p)$. It follows that 
\[
TU = \big(J_0S_0P_r - (c_r/c_s) J_0S_0P_s\big)J_rP_r  = J_0S_0P_r \notin \mathcal I(c),
\]
since 
\[x\mapsto \sum_{j=1}^\infty c_j[TU]_{0j}x=\sum_{j=1}^\infty c_j[J_0S_0P_r]_{0j}x=c_rS_0x\] is not an approximable operator $X\to Y$ as $c_r\neq 0$. This means that $\mathcal I(c)$ is not a right ideal of $\mathcal L(Z_p)$, whence $\mathcal I(c)$ is a non-trivial closed $\mathcal K(Z_p)$-subideal of $\mathcal L(Z_p)$.

\smallskip

(iii) Let $ \emptyset \neq B \varsubsetneq \mathbb N$ be an arbitrary subset. Pick $r\in B$ and $s\in supp(c)\setminus\{r\}$ and consider again the operator $T:=T_{(r,s)}\in\mathcal I(c)$ from \eqref{091725}. Since \[T_{0r}=P_0\big(J_0S_0P_r- (c_r/c_s) J_0S_0P_s\big)J_r=S_0\notin \mathcal A(X,Y)\] we obtain that $T\notin \mathcal I_B$. Thus $\mathcal I(c)\neq\mathcal I_B$.

\smallskip

(iv)  Suppose $r \in supp(c) \setminus supp(d)$, and let $R = J_0S_0P_r \in \mathcal K(Z_p)$, where $S_0$ is the operator from \eqref{fix}.
Clearly $R \in \mathcal I(d)$ as $R_{0s} = 0$ for all $s \in supp(d)\cup\{0\}$. On the other hand,
\[
x\mapsto\sum_{j=1}^\infty  c_jR_{0j}x= c_rS_0x
\]
is not approximable, so that  $R \notin \mathcal I(c)$. The case where  $s \in supp(d) \setminus supp(c)$ is symmetric.

\smallskip

(v) Note first that $\mathcal I(c)=\mathcal I(d)$ holds if $c=\lambda d$ for some non-zero scalar $\lambda$. 
In fact, if $S\in\mathcal I(d)$, then  $x\mapsto\sum_{j=1}^\infty d_jS_{0j}x$ defines an approximable operator $X\to Y$, so that
\[
x\mapsto \lambda\Big(\sum_{j=1}^\infty d_jS_{0j}x\Big)=\sum_{j=1}^\infty c_jS_{0j}x
\] 
is  also approximable. Thus $S\in \mathcal I(c)$. 
The reverse inclusion  $\mathcal I(c) \subset \mathcal I(d)$ is similar.

Conversely, assume  that  $c\neq \lambda d$ for all non-zero scalars $\lambda$. If $supp(c)\neq supp(d)$, then  $\mathcal I(c)\neq  \mathcal I(d)$ by part (iv).
Thus we may assume that $supp(c)=supp(d)$. Let $r\in supp(c)$, that is,  $c_r\neq 0$ and $d_r\neq 0$. 
It follows that  there is $s\in supp (c)$ so that 
\begin{equation}\label{1114}
\frac{c_r}{d_r}\neq \frac{c_s}{d_s},
\end{equation}
since otherwise $c=(c_r/d_r)d$, which contradicts the assumptions on $c$ and $d$.
Consider once again the operator
\[
T:=T_{(r,s)}=J_0S_0P_r - (c_r/c_s) J_0S_0P_s\in\mathcal I(c)
\]
from \eqref{091725}. 
In this case the operator 
\[
x\mapsto \sum_{j=1}^\infty d_jT_{0j}x=d_rS_0x-d_s(c_r/c_s) S_0x=\frac{d_rc_s-d_sc_r}{c_s}S_0x
\]
is not approximable $X \to Y$, since $S_0\notin \mathcal A(X,Y)$ and $d_rc_s-d_sc_r \neq 0$ in view of \eqref{1114}. 
Hence  $T\notin \mathcal I(d)$, which demonstrates that $\mathcal I(c)\neq \mathcal I(d)$.

\smallskip

The unit sphere  $\{c \in \ell^1: \Vert c\Vert_1 = 1\}$ is separable, so it follows that the family $\mathfrak F_+$ has the size of the continuum. 
Part (iii)  yields that  $\mathfrak F_+  \cap \mathfrak F = \emptyset$.

\smallskip

The argument of Theorem \ref{newsubids} will be complete once  we have established  the subsequent Lemma \ref{tech1}, which was 
used in the proof of part (i).
\end{proof}

The following technical fact is based on  \cite[Lemma 4.6]{TW22}.

\begin{lemma}\label{tech1}
Let $1 < p < \infty$ and $Z_p = \big( \bigoplus_{j=0}^\infty X_j\big)_{\ell^p}$, where $X_j$ is a Banach space for  $j \in \mathbb N_0$.
If $S = [S_{jk}] \in \mathcal K(Z_p)$ and $U  = [U_{jk}] \in \mathcal L(Z_p)$, then
\begin{equation}\label{kprod}
[US]_{0r} = \sum_{l = 0}^\infty U_{0l}S_{lr} 
\end{equation}
for each $r \in \mathbb N_0$, where the right hand series converges in the operator norm of  $\mathcal K(X_r,X_0)$. Moreover, 
 \eqref{kprod} also holds in case  $S\in\mathcal L(Z_p)$ and $U\in\mathcal K(Z_p)$.
\end{lemma}

\begin{proof}
Assume that $S\in\mathcal K(Z_p)$ and $U\in\mathcal L(Z_p)$. 
Let $Q_n = \sum_{s=0}^n J_sP_s$ be the natural projection of $Z_p$ onto $\bigoplus_{s=0}^n X_s \subset Z_p$ for any $n \in \mathbb N_0$. 
Since $S$ is a compact operator, it follows from the proof of  \cite[Lemma 4.6]{TW22}  that 
\[
\Vert Q_nS-S \Vert \to 0 \textrm{  as }  n\to\infty.
\]
Fix $r \in \mathbb N_0$, and observe that $\sum_{l=0}^n U_{0l}S_{lr}=P_0UQ_nSJ_r$ for each $n \in \mathbb N_0$. Since $[US]_{0r} = P_0USJ_r$
we get that 
\begin{equation}\label{20241107}
\Vert \sum_{l=0}^n U_{0l}S_{lr} - [US]_{0r}\Vert  
= \Vert P_0U(Q_nS - S)J_r \Vert \to 0
\end{equation}
as $n\to\infty$.

Next suppose that $S\in\mathcal L(Z_p)$ and $U\in\mathcal K(Z_p)$.
Since  $1 < p < \infty$ and $U\in\mathcal K(Z_p)$, it follows again  from   \cite[Lemma 4.6]{TW22}.  that
\[
\Vert UQ_n-U\Vert \to 0 \textrm{  as }  n\to\infty.
\]
An analogous estimate to \eqref{20241107}  yields that 
\eqref{kprod} holds, with convergence in the operator norm of  $\mathcal K(X_r,X_0)$.
\end{proof}

The  subideals $\mathcal I(c)$ have a somewhat stronger property than stated in part (i) of Theorem \ref{newsubids}. 
Recall for $n \ge 2$ that the Banach algebra $\mathcal A$ is said to be $n$-\textit{nilpotent} if 
$a_1 \cdots a_n = 0$ whenever  $a_j \in \mathcal A$ with $j = 1,\ldots, n$. The  closed ideal $\mathcal I$ of  $\mathcal A$ is 
 $n$-nilpotent if $\mathcal I$ is  $n$-nilpotent as a Banach algebra.

\begin{remarks}\label{091925}
(1) The closed ideals $q(\mathcal I(c))$ of  $\mathcal K(Z_p)/\mathcal A(Z_p)$ are $2$-nilpotent, 
where $q: \mathcal K(Z_p) \to \mathcal K(Z_p)/ \mathcal A(Z_p)$ denotes the quotient homomorphism. 
In fact, according to \cite[Lemma 4.6]{TW22} the products $SU \in \mathcal A(Z_p)$  whenever $S  \in \mathcal I(c)$  and $U \in \mathcal K(Z_p)$, 
since $[SU]_{jk}$ are approximable operators for each $j, k \in \mathbb N_0$  by the argument for part (i) of Theorem \ref{newsubids}.

\smallskip

(2) The reader may verify that for $c=(c_j)\in\ell^1$ the  sets 
\[
\mathcal J(c) = \{S = [S_{jk}] \in \mathcal K(Z_p): S_{00} \in \mathcal A(Y) \textrm{ and } \sum_{r=1}^\infty  c_rS_{0r}P_r \in \mathcal A(Z_p,Y)\}
\]
do not give any new closed subideals of $\mathcal L(Z_p)$, as  $\mathcal J(c) = \mathcal I_{supp(c)}$. The difference to \eqref{def0} is that
\[
\big(\sum_{r=1}^\infty  c_jS_{0r}P_r\big)x =  \sum_{r=1}^\infty  c_jS_{0r}x_r, \quad x = (x_j) \in Z_p,
\]
where $S_{0r}$ acts separately  on the component  $X_r = X$ of $Z_p$.
\end{remarks}

The result in \cite[Theorem 2.1]{TW23} suggests the question whether the  closed subideals $\mathcal I(c)$ from the family  $\mathfrak F_+$ are isomorphic as 
Banach algebras to the earlier subideals $\mathcal I_B$ from  $\mathfrak F$.
It turns out that, again unexpectedly, this is the case if condition \eqref{iso} holds.

We first recall the strategy behind  \cite[Theorem 2.1]{TW23}. Let $Z$ be a Banach space and suppose that $\mathcal A \subset \mathcal L(Z)$ and  
$\mathcal B \subset \mathcal L(Z)$ are closed subalgebras.
In order to show that $\mathcal A$ and $\mathcal B$ are isomorphic Banach algebras, it
suffices to find a linear isomorphism $U: Z \to Z$ such that the inner automorphism 
\begin{equation}\label{che}
\theta(S) = U^{-1}SU,  \quad  S \in \mathcal L(Z),
\end{equation}
of $\mathcal L(Z)$ satisfies $\theta (\mathcal A) \subset \mathcal B$ and $\theta^{-1} (\mathcal B) \subset \mathcal A$. 

We note for completeness that  \eqref{che} defines the only possible algebra isomorphism between non-zero closed subideals of $\mathcal L(Z)$.
Namely, suppose conversely that $\theta: \mathcal A \to \mathcal B$ is a Banach algebra isomorphism, where 
the closed subalgebras $\mathcal A$ and $\mathcal B$ of $\mathcal L(Z)$ satisfy  $\mathcal F(Z) \subset \mathcal A$ and $\mathcal F(Z)\subset \mathcal B$.
Then  a result of 
Chernoff \cite[Corollary 3.2]{Ch73} (see also \cite[Section 1.7.15]{P94}) implies that there is a linear
isomorphism $U: Z \to Z$ such that  \eqref{che} holds. 

\begin{theorem}\label{BAiso}
The closed subideals  $\mathcal I(c) \cong \mathcal I(d)$ for any non-zero sequences $c, d \in \ell^1$, and $\mathcal I(c) \cong \mathcal I_{\{j\}}$ for any $j \in supp(c)$.

In addition, if the component space $X$ in the definition \eqref{dir} of $Z_p$ satisfies condition \eqref{iso},  then 
\begin{equation}\label{alliso}
\mathcal I(c) \cong \mathcal I_B \  \textrm{ for any } \mathcal I(c) \in \mathfrak F_+ \textrm{  and }   \mathcal I_B \in \mathfrak F.
\end{equation}
\end{theorem}

\begin{proof}  The following claim contains  the essential step of the argument.

\smallskip

\textit{Claim 1}.  Suppose that  $c = (c_j) \in \ell^1$ is a non-zero sequence, and fix $j \in supp(c)$. Then 
\[
\mathcal I(c) \cong  \mathcal I_{\{j\}}.
\]

We define $V:Z_p\to Z_p$ by the conditions 
\[
[V(x)]_n=\begin{cases}
x_0&\text{if }n=0,\\
c_jx_j&\text{if } n=j,\\
x_n+c_nx_j&\text{if } n\neq j\text{ and }n\neq 0,
\end{cases}
\]
for $x = (x_k) \in Z_p$. Here we use the notation $[V(x)]_n = P_n(V(x)) \in X_n$ for the $n$:th coordinate of $V(x) \in Z_p$.

Observe first that $V$ is bounded on $Z_p$ since $c\in\ell^1$. In fact,  it is not difficult to check that
\[
V=I_{Z_p} - J_jP_j+\sum_{r=1}^\infty c_rJ_rI_{r,j}P_j
\]
where  $I_{r,j}:X_j\to X_r$ denotes the identity map $x\mapsto x$ (recall that $X_j = X_r = X$ for $j, r \in \mathbb N$).

It follows that  $V$ is a linear isomorphism $Z_p \to Z_p$, whose  inverse $V^{-1}$ is defined by
\[
[V^{-1}(x)]_n=\begin{cases}
x_0&\text{if }n=0,\\
c_j^{-1}x_j&\text{if } n=j,\\
x_n-c_nc_j^{-1}x_j&\text{if } n\neq j\text{ and }n\neq 0.
\end{cases}
\]
In other words, 
\[
V^{-1}= I_{Z_p}  +c_j^{-1}J_jP_j-c_j^{-1}\sum_{r=1}^\infty c_r J_r I_{r,j}P_j.
\]

We claim that 
\[
\theta(\mathcal I(c))\subset \mathcal I_{\{j\}} \textrm{  and } \theta^{-1}(\mathcal I_{\{j\}})\subset \mathcal I(c),
\]
where  $\theta$ is the inner automorphism of $\mathcal L(Z_p)$ defined by $T\mapsto V^{-1}TV$.
We will apply the simple observations that 
 \begin{equation}\label{fact1}
 P_0V=P_0V^{-1}=P_0 \textrm{ and } VJ_r=V^{-1}J_r=J_r \text{ for all }r\in\mathbb N_0\setminus\{j\}.
 \end{equation}

Suppose first that $T\in\mathcal I(c)$. It is easy to see that 
\[
[\theta(T)]_{00}=P_0V^{-1}TVJ_0=P_0TJ_0=T_{00}\in\mathcal A(Y). 
\]
Next note that 
\[
VJ_j=\big(I_{Z_p}-J_jP_j+\sum_{r=1}^\infty c_rJ_rI_{r,j}P_j\big)J_j=\sum_{r=1}^\infty c_rJ_rI_{r,j}.
\]
It follows with the help of \eqref{fact1} that 
\begin{align*}
x\mapsto [\theta(T)]_{0j}x &=P_0V^{-1}TVJ_jx=P_0T\big(\sum_{r=1}^\infty c_rJ_rI_{r,j}\big)x \\
&=\sum_{r=1}^\infty c_r P_0TJ_rx = \sum_{r=1}^\infty c_r T_{0r}x
\end{align*}
defines an approximable operator $X\to Y$, since $T\in\mathcal I(c)$ by assumption.  Thus $\theta(T)\in \mathcal I_{\{j\}}$.

Towards the second inclusion $\theta^{-1}(\mathcal I_{\{j\}})\subset \mathcal I(c)$, let $T\in\mathcal I_{\{j\}}$. 
Note first that the inverse map  $\theta^{-1}$ satisfies $\theta^{-1}(T) = VTV^{-1}$.  As above we get that  $[\theta^{-1}(T)]_{00}\in\mathcal A(Y)$.
By using  \eqref{fact1} we see  for $j \in \mathbb N$ that 
\begin{align*}
x\mapsto & \sum_{r=1}^\infty c_r[\theta^{-1}(T)]_{0r}x=\sum_{r=1}^\infty c_rP_0VTV^{-1}J_rx\\
&=c_jP_0TV^{-1}J_jx+\sum_{r\neq j} c_rP_0TV^{-1}J_rx\\
&=P_0T(0,-c_1x,\ldots,-c_{j-1}x,x,-c_{j+1}x,\ldots,-c_{n+1}x,\ldots)+\sum_{r\neq j} c_rP_0TJ_rx\\
&=P_0T\big(J_jx-\sum_{r\neq j}c_rJ_rx\big)+\sum_{r\neq j} c_rP_0TJ_rx\\
&=P_0TJ_jx
\end{align*}
defines an approximable operator $X \to Y$, as $T\in\mathcal I_{\{j\}}$ by assumption. Consequently, $\theta^{-1}(T)\in \mathcal I(c)$.
Conclude that $\mathcal I(c) \cong  \mathcal I_{\{j\}}$.

\smallskip

Suppose next that $c \in \ell^1$ and $d \in \ell^1$ are arbitrary non-zero sequences. Fix $j \in supp(c)$ and $k \in supp(d)$. By Claim 1 we know that 
$\mathcal I(c) \cong  \mathcal I_{\{j\}}$ and $\mathcal I(d) \cong  \mathcal I_{\{k\}}$. Moreover, from \cite[Lemma 2.2]{TW23} we get  that
$\mathcal I_{\{j\}} \cong  \mathcal I_{\{k\}}$, so that  $\mathcal I(c) \cong \mathcal I(d)$.  
Recall  for completeness  that if $\rho: \mathbb N \to \mathbb N$ is any permutation that
interchanges the pair $(j,k)$, then the linear isometry $U: Z_p \to Z_p$ defined by
\[
U(y,x_1,x_2, \ldots) = (y,x_{\rho(1)},x_{\rho(2)}, \ldots), \quad (y,x_1,x_2, \ldots) \in Z_p,
\]
induces a Banach algebra isomorphism $T \mapsto \psi(T) = U^{-1}TU$ between  $\mathcal I_{\{j\}}$ and $\mathcal I_{\{k\}}$.

Finally,  if the component space $X$ of $Z_p$ satisfies \eqref{iso}, then 
\[
\mathcal I_A \cong \mathcal I_{\{1\}} 
\]
for any  $\emptyset \neq A \varsubsetneq \mathbb N$ by \cite[Theorem 2.1]{TW23}, which implies  that \eqref{alliso} holds.
\end{proof} 

\begin{remarks}\label{thm34}
(1) Note  that  \eqref{iso} guarantees that
$\mathcal I_A \cong \mathcal I_{\{1\}}$ for infinite subsets $A \varsubsetneq \mathbb N$, which
is outside  Claim 1 of the above proof.

\smallskip

(2) It is somewhat surprising that $\mathcal I(c) \cong \mathcal I(d)$ for any non-zero sequences $c, d \in \ell^1$. 
For instance, consider $c = (1/j^2)$ and $d = (1/2^j)$, for which 
$\lim_{j\to\infty}c_j/d_j = \infty$. Theorem  \ref{BAiso} yields that the subideals
$\mathcal I(c) \cong \mathcal I(d)$, as well as $\mathcal I(\chi_A \cdot c) \cong \mathcal I(\chi_B \cdot d)$
for any non-empty sets $\emptyset \neq A, B \subset \mathbb N$. 
Here $\chi_A \cdot c$ denotes the sequence 
$(\chi_A(j) c_j)_{j\in \mathbb N}$ for $c = (c_j) \in \ell^1$ and  $\emptyset \neq A \subset \mathbb N$.
\end{remarks}

\smallskip

Write  $\mathcal A \ncong \mathcal B$ for non-isomorphic Banach algebras $\mathcal A$ and $\mathcal B$.
The following question  is motivated by the result that  there are closed $\mathcal S(L^p)$-subideals  $\mathcal J_1$ and $\mathcal J_2$ of $\mathcal L(L^p)$ 
 with $\mathcal J_1 \ncong \mathcal J_2$
 for $1 \le p \le \infty$ and $p \neq 2$, see \cite[Theorem 3.4 and Proposition 3.6]{TW23}. 

\smallskip

\noindent \textbf{Problem}. Does there exist a Banach space $X$ together with non-trivial closed $\mathcal K(X)$-subideals $\mathcal J_1$ and $\mathcal J_2$ of
$\mathcal L(X)$ for which  $\mathcal J_1 \ncong \mathcal J_2$?
 Is this possible  for $X = Z_p$?  (Note that $X$ has to fail the A.P. in  such an example.)
    
\smallskip

Towards this question we next  construct an increasing  sequence $(\mathcal M_r)$ of 
 closed subideals of  $\mathcal K(Z_p)$ for which $\mathcal M_r \not\cong \mathcal M_s$  for $r \neq s$. 
The argument is a variant of  \cite[Theorem 3.4.(ii)]{TW23}. In particular, 
the subalgebras $\mathcal M_r \not\cong \mathcal I$ for $\mathcal I \in \mathfrak F  \cup \mathfrak F_+$,
where $\mathfrak F$ and $\mathfrak F_+$ are the families of closed subideals of $\mathcal L(Z_p)$ defined by  \eqref{osubid} and \eqref{def0}.
The subalgebras  $\mathcal M_r$ are closed $3$-subideals of $\mathcal L(Z_p)$ in the subsequent terminology of Section \ref{nsubs},
but we do not know if they are actually closed subideals (that is, $2$-subideals) of $\mathcal L(Z_p)$.

\begin{example}\label{3subids}
Let $(X,Y)$ be any pair of Banach spaces that defines the direct sum $Z_p$ in \eqref{dir}.
Fix $S_0 \in  \mathcal K(X,Y) \setminus \mathcal A(X,Y)$, and  put
\[V_s = J_0S_0P_s  \in \mathcal I_{\{1\}}\] 
for $s \ge 2$, where $\mathcal I_{\{1\}}=\{T\in\mathcal K(Z_p) : T_{00}\in \mathcal A(Y), T_{01}\in\mathcal A(X,Y)\}$ is the closed $\mathcal K(Z_p)$-subideal from  \eqref{osubid}.
Let 
\[
\mathcal M_r = span(V_2,\ldots, V_r)   + \mathcal A(Z_p) \quad \textrm{ for }  r  \ge 2.
\]
Then the  following properties hold:

\begin{enumerate}
\item[(i)]\   $\mathcal M_r$ are closed subspaces of $\mathcal I_{\{1\}}$, and 
$\mathcal A(Z_p) \varsubsetneq \mathcal M_r \varsubsetneq \mathcal M_s   \varsubsetneq  \mathcal I_{\{1\}}$
for $2 \le r < s$. 

\smallskip

\item[(ii)]\ $\mathcal M_r  \subset \mathcal I_{\{1\}}$ is a closed ideal for $r \ge 2$, so that 
$(\mathcal M_r)$ is an increasing chain of 
closed $\mathcal I_{\{1\}}$-subideals of $\mathcal K(Z_p)$.

\smallskip

\item[(iii)]\   $\mathcal M_r \not\cong \mathcal M_s$ for any $r, s \ge 2$ with $r \neq s$.
\end{enumerate}
\end{example}

\begin{proof} (i)  It is straightforward to verify that $\mathcal M_r$ are closed subspaces of $\mathcal I_{\{1\}}$ for $r \ge 2$. We observe that $(V_2,\ldots, V_r)$ is linearly independent modulo $\mathcal A(Z_p)$. Namely, if
\[
R := \sum_{j=2}^r c_jV_j  \in \mathcal A(Z_p), 
\]
then $R_{0k} = P_0\big(\sum_{j=2}^r c_jV_j\big)J_k = c_kS_0 \in \mathcal A(X,Y)$, which 
implies that $c_k = 0$ for $k = 2,\ldots, r$ (recall that $S_0 \notin \mathcal A(X,Y)$).  In particular,
\[\mathcal A(Z_p) \varsubsetneq  \mathcal M_r \varsubsetneq \mathcal M_s \varsubsetneq \mathcal I_{\{1\}}\] for  $2 \le r < s$.
\smallskip

(ii)  Recall from  \cite[Theorem 4.5.(i)]{TW22} that
\begin{equation*}\label{right}
UV \in \mathcal A(Z_p) 
\end{equation*}
for any $U \in \mathcal I_{\{1\}}$ and $V \in \mathcal K(Z_p)$.  In particular, if $U \in \mathcal M_r \subset \mathcal I_{\{1\}}$ and $V \in \mathcal I_{\{1\}} \subset \mathcal K(Z_p)$, then
$UV \in \mathcal A(Z_p) \subset \mathcal M_r$ and $VU \in \mathcal A(Z_p) \subset \mathcal M_r$.
Thus  $\mathcal M_r$ is a closed $\mathcal I_{\{1\}}$-subideal of $\mathcal K(Z_p)$ for $r \ge 2$, since the sequence 
\[
\mathcal M_r \subset \mathcal I_{\{1\}} \subset \mathcal K(Z_p) 
\]
consists of  closed relative ideals.

\smallskip

(iii) Let $2 \le r  < s$, and suppose to the contrary that 
$\theta: \mathcal M_r  \to \mathcal M_s$ is a Banach algebra isomorphism. Since $\mathcal M_r$ and $\mathcal M_s$ are closed subalgebras of $\mathcal L(Z_p)$
that contain $\mathcal A(Z_p)$,
Chernoff's theorem \cite[Corollary 3.2]{Ch73} implies that there is a linear isomorphism $U: Z_p \to Z_p$ such that
\begin{equation}\label{auto}
S \mapsto \theta(S) = U^{-1}SU,\quad  S \in \mathcal M_r.
\end{equation}
The map $\theta$ from \eqref{auto}  defines a Banach algebra isomorphism $\mathcal L(Z_p) \to \mathcal L(Z_p)$, for which
$\theta(\mathcal M_r) = \mathcal M_s$ by definition and $\theta(\mathcal A(Z_p)) = \mathcal A(Z_p)$ as $\mathcal A(Z_p)$ is a closed ideal of $\mathcal L(Z_p)$.
Hence $\theta$ induces a well-defined Banach algebra isomorphism $\psi: \mathcal M_r/ \mathcal A(Z_p) \to \mathcal M_s/ \mathcal A(Z_p)$ through
\[
\psi(S + \mathcal A(Z_p)) = \theta(S) + \mathcal A(Z_p), \quad S + \mathcal A(Z_p) \in  \mathcal M_r/ \mathcal A(Z_p).
\]
In addition,  $\theta(\mathcal M_r) = \mathcal M_s$ yields that 
$\psi(\mathcal M_r/ \mathcal A(Z_p)) = \mathcal M_s/ \mathcal A(Z_p)$.

This means that  the quotients $\mathcal M_r/ \mathcal A(Z_p)$ and $\mathcal M_s/ \mathcal A(Z_p)$ are linearly isomorphic. 
However, this is not possible as 
$dim\big(\mathcal M_r/ \mathcal A(Z_p)\big) = r - 1$
for $r \ge 2$  by the proof of part (i). Conclude that  $\mathcal M_r \not\cong \mathcal M_s$ for $r \neq s$.
\end{proof}

\begin{remark}\label{WST}
Let $X$ be a Banach space for which the algebra $\mathcal K(X)/\mathcal A(X)$ is not nilpotent. Such spaces $X$ exists by \cite[Theorem 2.9]{TW22}. In fact, there are even closed subspaces $X \subset \ell^p$ with this property for $p \neq 2$. 
(See also  \cite[Example 5.24 and Corollary 5.25]{ST11}  as well as \cite[Proposition 3.1]{TW21} for alternative constructions.)

An abstract technique, due to Wojtynski \cite[Theorem]{Wo78} and extended by Shulman and Turovskii \cite[Theorem 5.22 and Corollary 5.23]{ST11},  produces 
an infinite chain $(\mathcal I_{\alpha})$ of closed ideals of $\mathcal L(X)$ such that $\mathcal A(X) \varsubsetneq \mathcal  I_{\alpha} \varsubsetneq \mathcal K(X)$ for all $\alpha$. 
This result applies to $Z_p$ if one replaces the space $Y$ in definition \eqref{dir} by $Y \oplus Y_0$, for which the quotient $\mathcal K(Y_0)/\mathcal A(Y_0)$ is not nilpotent. 
In contrast,  the families $\mathfrak F$ and $\mathfrak F_+$ of non-trivial closed $\mathcal K(Z_p)$-subideals from this section are very explicit.
\end{remark} 


\section{Closed $n$-subideals of $\mathcal L(X)$}\label{nsubs}

In this section we study closed $n$-subideals of the Banach algebras $\mathcal L(X)$ for  Banach spaces $X$, which provides a graded extension of the concept 
of a closed subideal.  For each $n \ge 2$ we uncover various Banach spaces $X$ for which   there are 
closed $(n+1)$-subideals of $\mathcal L(X)$ that fail to be an $n$-subideal (Theorems \ref{general}, \ref{nsubids2} and \ref{nsubids3}).  
One of our main results  (Theorem \ref{nsubidchain}) constructs  
an infinite decreasing chain $(\mathcal M_n)$ of such closed $(n+1)$-subideals of $\mathcal L(X)$ for certain spaces $X$. 
Moreover, Theorem \ref{nsubid} and Proposition \ref{nsubidqn} exhibit  non-trivial closed $n$-subideals contained in the compact operators $\mathcal K(X)$ 
(in these  cases $X$ fails the A.P).

Let  $\mathcal A$ be a Banach algebra  over the real or complex scalars, $ \mathcal I\subset \mathcal A$ be a subalgebra and  $n\in\mathbb N$.
We say that $ \mathcal I$ is an (algebraic) $n$-\textit{subideal}  of $\mathcal A$ 
if there is a sequence $\mathcal J_0,\ldots,  \mathcal J_n$ of subalgebras of $\mathcal A$ such that
\begin{equation}\label{2508}
 \mathcal I=  \mathcal J_n\subset\cdots\subset  \mathcal J_1\subset  \mathcal J_0 = \mathcal A,
 \end{equation}
where $ \mathcal J_k$ is an ideal of $ \mathcal J_{k-1}$ for each $k=1,\ldots,n$. 
The  $n$-subideal  $\mathcal I$ of $\mathcal A$  is  \textit{closed},  if in \eqref{2508} the subalgebras 
$\mathcal I$ and $\mathcal J_k$, for $k=1,\ldots, n$, are  closed in $\mathcal A$ (note that by continuity it suffices that $\mathcal I$ is closed in $\mathcal A$).
In this event we will often say that \eqref{2508} is a chain of closed relative ideals associated to  $\mathcal I$.

We put
\[
\mathcal{SI}_n(\mathcal A) = \{ \mathcal I :  \mathcal I \textrm{ is a closed $n$-subideal of } \mathcal A\}.
\]
Clearly the closed 1-subideals $\mathcal I\subset \mathcal A$ are the closed ideals of $\mathcal A$, 
and the  closed $2$-subideals correspond to the closed subideals discussed in Sections \ref{Tarbard} and \ref{newsubs},  as well as in \cite{TW23}.
Condition  \eqref{2508} allows that  $\mathcal J_k = \mathcal J_{k-1}$  for some indices $k$, so that  
$\mathcal{SI}_n(\mathcal A)  \subset \mathcal{SI}_m(\mathcal A)$ for $m > n$. However, if $\mathcal I$ is a closed $m$-subideal of $\mathcal A$ for some $m \in \mathbb N$, then 
there is a minimal $n \in \mathbb N$ such that $\mathcal I \in \mathcal{SI}_n(\mathcal A)$. 

\begin{remarks}
(1) Algebraic $n$-subideals  of Banach algebras for $n \ge 2$ appear in \cite[2.2.1]{ST14} (and earlier implicitly 
in  \cite[Theorem 2.24]{ST05}). 
For Lie algebras there is a parallel concept of $n$-step Lie subideals, see e.g. \cite{Ste70}. 
The latter  notion concerns  the Lie product $[x,y] = xy - yx$ on $\mathcal A$ when applied to a Banach algebra $\mathcal A$.

For the Banach algebras $\mathcal L(X)$ the corresponding Lie algebraic concepts are very different from the (sub)ideals studied in this paper.
Let $H$ be a complex separable infinite-dimensional Hilbert space.
It is known that $span(I_H)$, $\mathcal K(H)$ and $span(I_H) + \mathcal K(H)$ 
are precisely  the proper closed Lie ideals of $\mathcal L(H)$, see \cite[Corollary 3]{FMS82}. 
Furthermore,  there are non-trivial (non-closed) $2$-subideals $\mathcal J \subset \mathcal L(H)$ that are Lie ideals of $\mathcal L(H)$, see \cite[Theorem]{V90}.

(2) Ring theoretic subideals appeared earlier in ring theory (see the discussion before Proposition \ref{andru}), though the term subideal
is not in common use.
\end{remarks}

Any non-zero closed $n$-subideal of the Banach algebras  $\mathcal L(X)$  must contain the approximable operators.

\begin{proposition}\label{propA}
Let $X$ be a  Banach space and $n\in\mathbb N$. If $\mathcal I$ is a non-zero $n$-subideal of $\mathcal L(X)$, then $\mathcal F(X)\subset \mathcal I$.
In particular, if $\mathcal I \subset \mathcal L(X)$ is a non-zero closed $n$-subideal, then $\mathcal A(X) \subset \mathcal  I$.
\end{proposition}

\begin{proof}
The claim is verified by induction. The case $n=1$ is contained in Remark \ref{nonzero}, or  e.g. \cite[Theorem 2.5.8.(i)]{D00}.

Let $n\in\mathbb N$ and assume that $\mathcal  F(X)\subset \mathcal J$ holds for every non-zero $n$-subideal $\mathcal  J$ of $\mathcal L(X)$. 
Let $\mathcal I$ be a non-zero $(n+1)$-subideal of $\mathcal L(X)$.
It suffices to check that $x^*\otimes x\in \mathcal  I$ for any $x^*\in X^*$
and $x\in X$. 

Towards this, let $T\in \mathcal  I$ be a non-zero operator, and pick $y^*\in X^*, y\in X$ such that $y^*(Ty)=1$. Since $\mathcal  I$ is an $(n+1)$-subideal of $\mathcal L(X)$,
there is an $n$-subideal $\mathcal J$ of $\mathcal L(X)$ such that $\mathcal  I \subset \mathcal  J$ and $\mathcal I$ is an ideal of $\mathcal  J$. 
By assumption $\mathcal  J$ contains all rank-one operators, so that 
\begin{equation*}
x^*\otimes x=y^*(Ty)(x^*\otimes x)=(y^*\otimes x)T(x^*\otimes y)\in \mathcal  I.\qedhere
\end{equation*}
\end{proof}

It is a natural  problem for the class of Banach algebras $\mathcal A = \mathcal L(X)$ whether 
\begin{equation}\label{strict}
\mathcal{SI}_n(\mathcal L(X))  \varsubsetneq  \mathcal{SI}_{n+1}(\mathcal L(X))
\end{equation}
holds for any given $n \ge 2$, that is, does $\mathcal L(X)$ contain closed $(n+1)$-subideals $\mathcal I$ that fail to be $n$-subideals.
 Various examples for  $n = 1$ can be found in \cite{TW23}, as well as in Sections \ref{Tarbard} and \ref{newsubs}, 
 but the cases $n \ge 2$ require additional ingredients.
 
We first observe that if the closed subalgebra $\mathcal I \subset \mathcal L(X)$  witnesses  \eqref{strict} for some $n \in \mathbb N$, then   
$\mathcal I$ fails to have an approximate identity. 
Recall that the net $(e_{\alpha}) \subset \mathcal I$ is an \textit{approximate identity} for $\mathcal I$ if for all $x \in \mathcal I$ one has
\[
 \lim_{\alpha} e_{\alpha}x = x =  \lim_{\alpha} xe_{\alpha}. 
\]
The case $n = 1$ is discussed in  \cite[Lemma 3.1]{TW23}.

\begin{proposition}\label{ai}
Suppose that $\mathcal I$ is a closed $n$-subideal of the Banach algebra $\mathcal A$, where $n \ge2$. Let 
\begin{equation}\label{nsub}
\mathcal I = \mathcal J_n \subset \mathcal J_{n-1}\subset \cdots \subset \mathcal J_1 \subset \mathcal J_0=\mathcal A
\end{equation}
 be an $n$-chain of closed relative ideals. 
  If $\mathcal J_k$ or $\mathcal J_{k-1}$ has an approximate identity for some
 $k \in \{2,\ldots,n\}$, then $\mathcal I$ is a closed $(n-k+1)$-subideal of $\mathcal A$. 
 
 In particular, if $\mathcal I$ or $\mathcal J_{n-1}$ has an approximate identity, then
 $\mathcal I$ is a closed ideal of $\mathcal A$.
 \end{proposition} 
 
 \begin{proof}
 Let $k \in  \{2,\ldots,n\}$ be fixed. If  $\mathcal J_k$ has an approximate identity, then \cite[Lemma 3.1]{TW23} implies that 
 $\mathcal J_k \subset \mathcal J_{k-2}$ is a closed ideal, so that the closed subalgebra $\mathcal J_{k-1}$ is redundant in  \eqref{nsub}.
 By iterating this procedure we get that $\mathcal J_k$ is a closed ideal of $\mathcal A$. 
 
The same argument shows that, if $\mathcal J_{k-1}$ has an approximate identity, then $\mathcal J_{k-1}$ is closed ideal of $\mathcal A$.  Another application of  \cite[Lemma 3.1]{TW23} yields that $\mathcal J_k$ is a closed ideal of $\mathcal A$. 

Thus in both cases  $\mathcal I =  \mathcal J_n$ is a closed $(n-k+1)$-subideal of $\mathcal A$.
Note that if $\mathcal I$ or $\mathcal J_{n-1}$ has an approximative identity, then $\mathcal I$ is a closed ideal of $\mathcal A$ by choosing $k=n$ above.
  \end{proof}

\begin{remarks}\label{sect4}
(1) If \eqref{strict} holds for some $n \ge 2$, then according to the definition one has $\mathcal{SI}_{k}(\mathcal L(X))\varsubsetneq \mathcal{SI}_{k+1}(\mathcal L(X))$ for $k = 1,\ldots, n-1$. 
Moreover, if $\mathcal{SI}_n(\mathcal L(X)) = \mathcal{SI}_{n+1}(\mathcal L(X))$, then $\mathcal{SI}_m(\mathcal L(X)) = \mathcal{SI}_n(\mathcal L(X))$ for all $m > n$.

\smallskip

(2) Let  $\mathcal A$ be a $C^*$-algebra, and suppose that $\mathcal I \subset \mathcal A$ is a closed $n$-subideal for some $n \ge 2$. Then Proposition \ref{ai} implies 
that $\mathcal I$ is already a closed ideal of $\mathcal A$. Namely, any closed ideal $\mathcal J$ of a $C^*$-algebra is itself a $C^*$-algebra and
has an approximate identity, see e.g. \cite[Theorem 3.2.21]{D00}.

\smallskip

(3)  We refer to Example \ref{nonil2}, and the discussion preceding it, for examples of spaces $X$  and  closed subalgebras  
$\mathcal I\subset\mathcal L(X)$ that fail to be $n$-subideals of $\mathcal L(X)$ for any $n \in \mathbb N$.
\end{remarks}
 
Let $\mathcal  A$ be a Banach algebra, and  $\mathcal  J \subset \mathcal  A$ be a closed ideal. For $k \ge 2$ put 
\[
\mathcal J^k: = span\{a_1  \cdots a_k \mid  a_1,\ldots, a_k\in \mathcal J\}.
\]
The power  ideals $\mathcal J^k$ need not be closed in $\mathcal A$ for $k \ge 2$, and there are relevant examples among the Banach algebras 
$\mathcal A \subset \mathcal L(X)$ for suitable spaces $X$.
Let $(\mathcal N(X), \Vert \cdot \Vert_{\mathcal N})$  be the Banach operator  ideal of the nuclear operators on $X$.
Dales and Jarchow \cite[Lemma 3.4]{DJ94} observed that
$\mathcal N(X)^2$ has infinite codimension in $\mathcal N(X)$ for any infinite-dimensional Banach space $X$, whereas 
$\mathcal N(X)^2$ is $\Vert \cdot \Vert_{\mathcal N}$-dense in $\mathcal N(X)$.
Similar properties hold for $\mathcal A(P)^2 \subset \mathcal A(P)$ by \cite[Theorem 4.2]{DJ94}, where $P$ is a Pisier space. 
For general Banach algebras $\mathcal A$ there are many further examples, see e.g. \cite[p. 198]{D77}.

\smallskip
 
 The following result contains  a general principle underlying many of  our examples of 
 closed $(n+1)$-subideals that fail to be  $n$-subideals. Let $X$ be a Banach space, and $\mathcal I \subset \mathcal L(X)$. 
 We will use the operator matrix notation
\[
 \begin{bmatrix}
\mathcal I&\mathcal I\\\mathcal I&\mathcal I
\end{bmatrix} =
\left\{ \begin{bmatrix}
S_{11}&\mathcal S_{12}\\S_{21}& S_{22}
\end{bmatrix} : S_{ij} \in \mathcal I \textrm{ for } i, j = 1,2 \right\},
\]
to define classes of operators  $S = \begin{bmatrix}
S_{11}&\mathcal S_{12}\\S_{21}& S_{22}
\end{bmatrix}$ on the direct sum $X \oplus X$.

\begin{theorem}\label{general}
Let $n\in\mathbb N$. Suppose that $X$ is a Banach space for which there is a closed ideal  $\mathcal I\subset \mathcal L(X)$ and an operator $S\in \mathcal I$ 
such that 
\begin{equation}\label{crit}
S^n\notin \overline{\mathcal I^{n+1}}.
\end{equation} 
Then  $Z=X\oplus X$ admits a non-trivial closed $(n+1)$-subideal 
\[
\mathcal M \in \mathcal{SI}_{n+1}(\mathcal L(Z)) \setminus \mathcal{SI}_{n}(\mathcal L(Z)).
\]
\end{theorem}

\begin{proof}
Let
\[
\mathcal J=\begin{bmatrix}
\mathcal I&\mathcal I\\\mathcal I&\mathcal I
\end{bmatrix} \subset \mathcal L(Z).
\]
It is straightforward to check that $\mathcal J$ is  a closed ideal of $\mathcal L(Z)$.  
Put
\[T=\begin{bmatrix}
S&0\\0&0
\end{bmatrix}\in \mathcal J,
\]
and define
\begin{equation}\label{nsubgen}
\mathcal M_k  :=span(T,\ldots,T^k)+\overline{\mathcal J^{k+1}} \subset \mathcal J
\end{equation}
for $1 \le k \le n$ and $\mathcal M_0: = \mathcal J$. We first claim that 
\begin{equation}\label{newchain}
\mathcal M_n \subset \mathcal M_{n-1} \subset \cdots \subset  \mathcal M_1 \subset \mathcal J \subset \mathcal L(Z)
\end{equation}
is a chain of closed relative ideals, so that $\mathcal M_n \in \mathcal{SI}_{n+1}(\mathcal L(Z))$.

In fact, each $\mathcal M_k$ is a closed subspace of $\mathcal J$ as a finite-dimensional extension of the closed subspace $\overline{\mathcal J^{k+1}}$ of $\mathcal L(Z)$.
Moreover, $\mathcal M_k \subset \mathcal M_{k-1}$ for $k = 1,\ldots, n$, since $T^k \in \mathcal J^k$ and $\overline{\mathcal J^{k+1}} \subset \overline{\mathcal J^{k}}$.
We next verify that $\mathcal M_k$ is an ideal  of $\mathcal M_{k-1}$.
Towards this fix $k \in \{1,\ldots,n\}$ and let 
\[
U = \sum_{r=1}^k a_rT^r + U_0 \in \mathcal M_k, \quad V = \sum_{s=1}^{k-1} b_sT^s + V_0 \in \mathcal M_{k-1}
\]
be arbitrary, where $U_0 \in \overline{\mathcal J^{k+1}}$ and $V_0 \in \overline{\mathcal J^{k}}$. We claim  that 
\[
UV  =  \sum_{r=1}^k\sum_{s=1}^{k-1} a_rb_s T^{r+s}  +   U_0\big( \sum_{s=1}^{k-1} b_sT^s +V_0 \big) +  \big(\sum_{r=1}^k a_rT^r\big)V_0    \in \mathcal M_k.
\]
To see this, note that  the first term splits as 
\[
 \sum_{r=1}^k\sum_{s=1}^{k-1} a_rb_s T^{r+s}  := U_1 + V_1 \in span(T,\ldots,T^k)+\mathcal J^{k+1}\subset\mathcal M_k, 
\]
since $T^j \in \mathcal J^{k+1}$ for any $j = k+1,\ldots, 2k-1$. Moreover, 
\[
U_0\big( \sum_{s=1}^{k-1} b_sT^s +V_0 \big) \in \overline{\mathcal J^{k+1}}\subset \mathcal M_{k}
\] 
as $\overline{\mathcal J^{k+1}}$ is an ideal of $\mathcal L(Z)$. Finally,  
$T^jV_0 \in  \overline{\mathcal J^{k+1}}$ for $j = 1,\ldots, k$ by approximation, so that 
$\big(\sum_{r=1}^k a_rT^r\big)V_0  \in  \overline{\mathcal J^{k+1}}\subset\mathcal M_k$.  

The fact that 
$VU \in  \mathcal M_k$ is verified in an analogous manner. 
Altogether  \eqref{newchain}  implies that  $\mathcal M_n \in \mathcal{SI}_{n+1}(\mathcal L(Z))$ is a closed $(n+1)$-subideal. 

\smallskip

We next claim that $\mathcal M_n\notin \mathcal{SI}_n(\mathcal L(Z))$. In view of  Lemma \ref{newlemma} below it suffices to verify that
\begin{equation}\label{1409} 
\begin{bmatrix}
0&0\\
S^n&0
\end{bmatrix} =
\begin{bmatrix}
0&0\\
I_X&0
\end{bmatrix}\circ T^n\notin  \mathcal M_n.
\end{equation}
Suppose to the contrary that $\begin{bmatrix}
0&0\\
S^n&0
\end{bmatrix}\in\mathcal M_n$. Hence  there are scalars $a_1,\ldots, a_n$ 
and  $U = \begin{bmatrix}
U_{11}&U_{12}\\
U_{21}&U_{22}
\end{bmatrix} \in  \overline{\mathcal J^{n+1}}$ with 
\[
\begin{bmatrix}
0&0\\
S^n&0
\end{bmatrix} = \sum_{k=1}^n a_kT^k + U = \begin{bmatrix}
\sum_{k=1}^n a_kS^k+U_{11}&U_{12}\\
U_{21}&U_{22}
\end{bmatrix}.
\]
On the other hand, one verifies  by induction that 
\[
\overline{\mathcal J^{n+1}}=\begin{bmatrix}
\overline{\mathcal I^{n+1}}&\overline{\mathcal I^{n+1}}\\\overline{\mathcal I^{n+1}}&\overline{\mathcal I^{n+1}}
\end{bmatrix}.
\]
By combining these facts we get  that the component $S^n = U_{21} \in \overline{\mathcal I^{n+1}}$, 
which  contradicts the assumption \eqref{crit} of the theorem.
\end{proof}
 Above we made use of the following basic lemma, which we state separately for explicit reference.
 \begin{lemma}\label{newlemma}
 Let $n\in\mathbb N$ and suppose that $\mathcal M_n$ is a closed subalgebra of the Banach algebra $\mathcal A$. If $\mathcal M_n\in\mathcal{SI}_{n}(\mathcal A)$, then $s\cdot t^n\in \mathcal M_n$ for all $s\in\mathcal A$ and $t\in\mathcal M_n$.
 \end{lemma}
 \begin{proof}
 Let
 \[\mathcal M_n=\mathcal J_n\subset \cdots\subset \mathcal J_1\subset\mathcal A\]
 be a chain of closed relative ideals associated to $\mathcal M_n$. Since $s\cdot t\in\mathcal J_1$ and $t\in\mathcal J_k$ for all $k=1,\ldots,n$, by using the relative ideal property at each step of the chain, we obtain that
 $s\cdot t^n\in \mathcal M_n$.
 \end{proof}

Conceptually the most elementary examples of non-trivial closed $(n+1)$-subideals of $\mathcal L(X)$  that fail to be $n$-subideals can be defined  on  
finite direct sums of different $\ell^{p}$-spaces.  For $n \in \mathbb N$,  let  $1 \le  p_1<p_2<\cdots<p_{n+1} \le  \infty$ and define 
\begin{equation}\label{elem}
X_n = \bigoplus_{j=1}^{n+1} \ell^{p_j}. 
\end{equation}
For unicity of notation we use the convention that $\ell^{\infty} = c_0$ if $p_{n+1} = \infty$. 
We equip $X_n$ with the max-norm, though our results are also valid for equivalent norms.
The strictly singular operators $\mathcal S(X_n)$ will be the closed reference ideal $\mathcal I$ of $\mathcal L(X_n)$ in this application of 
Theorem  \ref{general}. We will repeatedly  use  the following facts.

For any pair of different spaces $Y$ and $Z$  from $\{\ell^p: 1 \le p < \infty\} \cup \{c_0\}$ 
we have
\begin{equation}\label{prop1}
\mathcal L(Y,Z) = \mathcal S(Y,Z) 
\end{equation}
by the total incomparability
of these spaces, see e.g.   \cite[p. 75]{LT77}. If  $X = \ell^q$ with $q > p$ or $X = c_0$, then  by Pitt's theorem
\begin{equation}\label{prop2}
\mathcal L(X, \ell^p) = \mathcal K(X, \ell^p),
\end{equation}
see e.g. \cite[Proposition 2.c.3]{LT77}. Moreover, 
\begin{equation}\label{prop3}
\mathcal S(Y) = \mathcal K(Y)
\end{equation}
for any $Y$ from $\{\ell^p: 1 \le p < \infty\} \cup \{c_0\}$, see e.g.  \cite[p. 76]{LT77} or \cite[Sections 5.1-5.2]{Pi80}.

Properties \eqref{prop1} - \eqref{prop3} imply that $S =  [S_{jk}] \in \mathcal S(X_n)$
if and only if $S_{jk}$ is compact for $1 \le j \le k \le n+1$ and $S_{jk}$ is bounded  for $1 \le k < j \le n+1$. In other words, the 
operator matrix of $S \in \mathcal S(X_n)$ modulo $\mathcal K(X_n)$ is lower subtriangular and has  a $0$-diagonal.
 
 \begin{theorem}\label{nsubids2}
Let $X_n$ be the space from   \eqref{elem}, where we fix $n \in \mathbb N$ and 
$1 \le  p_1< \cdots<p_{n+1} \le  \infty$. Then the following properties hold: 
\begin{enumerate}
\item[(i)]  the quotient algebra $\mathcal S(X_n)/ \mathcal K(X_n)$ is $(n+1)$-nilpotent, and
\[
\overline{\mathcal S(X_n)^{n+1}} = \mathcal K(X_n),
\]

\smallskip

\item[(ii)] there is  $S \in \mathcal S(X_n)$ such that  $S^n  \notin \mathcal K(X_n)$,

\smallskip

\item[(iii)]  there is a closed subalgebra  $\mathcal M \subset \mathcal S(X_n)$
for which 
\[
\mathcal M \in \mathcal{SI}_{n+1}(\mathcal L(X_n)) \setminus \mathcal{SI}_{n}(\mathcal L(X_n)).
\]
\end{enumerate}
\end{theorem} 

\begin{proof}
(i)  We claim that the products 
\begin{equation}\label{prodn} 
S_{n+1} \cdots S_1 \in \mathcal K(X_n)
\end{equation}
whenever $S_j \in \mathcal S(X_n)$ for $j = 1,\ldots, n+1$. In particular, this will imply that  $\mathcal S(X_n)^{n+1} \subset \mathcal K(X_n)$, so 
 the  quotient algebra $\mathcal S(X_n)/ \mathcal K(X_n)$ is $(n+1)$-nilpotent. Moreover, 
$\overline{\mathcal S(X_n)^{n+1}} = \mathcal K(X_n)$ by Remark \ref{nonzero},
 since $\overline{\mathcal S(X_n)^{n+1}}$ is a non-zero closed ideal of $\mathcal L(X_n)$
and $X_n$ has the A.P.  

By iterating \eqref{productofoperators} we obtain that for any $j, k \in \{1,\ldots,n+1\}$ the component $[S_{n+1} \cdots S_1]_{jk}$ of the product in \eqref{prodn}  is a finite sum of terms of  the form
\[
U_{n+1} \cdots U_1,
\]
where $U_r \in \mathcal S(\ell^{k_r},\ell^{k_{r+1}})$ and $k_r, k_{r+1} \in \{p_1,\ldots,p_{n+1}\}$ for $r = 1, \ldots,n+1$.
Thus we pick  $k_1,\ldots, k_{n+2}$ from the set  $\{p_1,\ldots,p_{n+1}\}$ containing $n+1$ elements, so there are 
$r < s$ such that  $k_r = k_s$. Since the strictly singular operators $\mathcal S$ form an operator ideal we obtain that the partial product 
\[
U_{s-1} \cdots U_r \in \mathcal S(\ell^{k_r}) =  \mathcal K(\ell^{k_r}) 
\]
in view of  \eqref{prop3}. Thus  $U_{n+1} \cdots U_1 \in \mathcal K(\ell^{k_1},\ell^{k_{n+2}})$, so that \eqref{prodn} holds. 

\smallskip

(ii) Recall that $p_k < p_{k+1}$ for $k = 1,\ldots , n$, and let  $i_k$ be the corresponding inclusion map  $\ell^{p_k} \to \ell^{p_{k+1}}$. 
(If $p_{n+1} = \infty$, then $i_n$ is the inclusion map $\ell^{p_n} \to c_0$.) We define $S \in \mathcal L(X_n)$ by 
 \[
 S =  \sum_{k=1}^{n}J_{k+1}i_kP_k,
 \] 
 where $J_k$ and $P_k$ are the canonical inclusions, respectively projections, related to the component space $\ell^{p_k}$ of the 
  finite direct sum $X_n$  for  $k = 1,\ldots, n+1$. 
Property \eqref{prop1} implies that $S  \in \mathcal S(X_n)$. 
 The representation 
 \begin{equation}\label{092225}
 Sx = (0,i_1x_1,\ldots,i_nx_n), \quad x = (x_1,\ldots,x_{n+1}) \in X_n,
 \end{equation}
 of $S$ acting as a shift operator on the components of $X_n$  implies that 
 \[
 S^{n}x=(0,\ldots,0,i_n \cdots  i_1 x_1), \quad x\in X_n.
 \]
 Thus $P_{n+1}S^nJ_1$ is the non-compact inclusion  $\ell^{p_1} \to \ell^{p_{n+1}}$, so that
 \[
 S^n  \notin \mathcal K(X_n) = \overline{\mathcal S(X_n)^{n+1}} 
 \]
in view of  part (i). 
\smallskip

(iii) Let $Z_n = X_n \oplus X_n$.  By parts (i) - (ii) and Theorem \ref{general} there is a closed subalgebra $\mathcal M_0 \subset \mathcal S(Z_n)$ such that 
\[
\mathcal M_0 \in \mathcal{SI}_{n+1}(\mathcal L(Z_n)) \setminus \mathcal{SI}_{n}(\mathcal L(Z_n)).
\]
Recall that $X \oplus X \approx X$ for  the spaces $X$ from $\{\ell^p: 1 \le p < \infty\} \cup \{c_0\}$. Thus there is a linear isomorphism
$U: Z_n \to X_n$, which induces a Banach algebra isomorphism 
\[
T \mapsto \theta(T) = UTU^{-1}, \quad \mathcal L(Z_n) \to \mathcal L(X_n).
\]
Note that  $\theta(\mathcal S(Z_n)) = \mathcal S(X_n)$,  and it is easy to check that algebra isomorphisms preserve the $n$-subideal structure
(cf. Lemma \ref{qsubid} for $n = 2$). It follows that the closed subalgebra
$\mathcal M = \theta(\mathcal M_0) \subset \mathcal S(X_n)$ is the desired example associated to  the space $X_n$. 
(Note that $\mathcal M$ will depend on the choice of  $U$.)
\end{proof}

A central aim of this section is to construct an infinite decreasing sequence of non-trivial closed $n$-subideals in $\mathcal L(X)$, where $X$ is an infinite direct sum of $\ell^{p_j}$-spaces. 
This answers positively a query  of Niels Laustsen. We remark that it is significantly  easier to find 
a space $Z$ for which there are closed subideals $\mathcal K_n \in \mathcal{SI}_{n+1}(\mathcal L(Z)) \setminus \mathcal{SI}_{n}(\mathcal L(Z))$
for each $n \in \mathbb N$, see the proof of part (ii) of Theorem \ref{nsubid}. In that construction  $(\mathcal K_n)_{n \in \mathbb N}$ will not be a chain.

Let $(p_j)_{j\in \mathbb N}$ be any strictly increasing sequence in $[1,\infty)$.
We put $E = \ell^r$ for $r$ satisfying $\sup_j p_j \le r \le \infty$, where the notational convention is again that $\ell^\infty = c_0$.
The desired Banach space will be the $E$-direct sum 
\begin{equation}\label{direct}
X = \big(\bigoplus_{j=1}^\infty \ell^{p_j}\big)_E  \ .
\end{equation}
We will find a decreasing chain $(\mathcal M_n)$ of closed $(n+1)$-subideals of $\mathcal L(X)$ such that \[\mathcal K(X)\subset\mathcal M_n\subset \mathcal S(X)\] for all $n\in\mathbb N$. Such an infinite chain lies outside the scope of Theorem \ref{general} 
as well as the subsequent Proposition \ref{lemma}. In particular, the quotient algebra $\mathcal S(X)/\mathcal K(X)$ is non-nilpotent in view of the finite shift operators from \eqref{092225}. This obstruction creates serious technical difficulties, which are resolved in the proofs of \eqref{claim1} and \eqref{claim2}.

 \begin{theorem}\label{nsubidchain}
 Let $X$ be the space in  \eqref{direct}. Then there is a strictly decreasing chain $(\mathcal M_n)_{n\in \mathbb N_0}$ of closed subalgebras of $\mathcal L(X)$ such that 
 $\mathcal M_0 = \mathcal S(X)$, and 
 $\mathcal M_n  \subset \mathcal M_{n-1}$ is a closed ideal for all $n \in \mathbb N$, so that 
 \[
\mathcal M_n \in \mathcal{SI}_{n+1}(\mathcal L(X)) \setminus \mathcal{SI}_{n}(\mathcal L(X)) 
 \]
 for all $n \in \mathbb N$. 
\end{theorem} 

\begin{proof}
 For each $n \in \mathbb N$ we consider the finite $E$-direct sum 
 \[
 X_n=\ell^{p_1}\oplus\cdots\oplus\ell^{p_{n+1}}.
 \]
 We use  $J_n:X_n\to X$ and $P_n:X\to X_n$ for the natural inclusions and the projections
 related to $X_n$.  
 
According to the proof of part (ii) of Theorem \ref{nsubids2}
there is   $S_n\in \mathcal S(X_n)$ such that $S_n^{n}\notin \mathcal K(X_n)$. 
Next, for each $n\in \mathbb N$ let $\theta_n:X_n\to X_n\oplus X_n$ be the linear isomorphism defined by 
\[
\theta_n(x_1,\ldots,x_{n+1}) := (\phi_1(x_1)_1,\ldots, \phi_{n+1}(x_{n+1})_1)+(\phi_1(x_1)_2,\ldots,\phi_{n+1}(x_{n+1})_2),
\]
where $\phi_k:\ell^{p_k}\to \ell^{p_k}\oplus \ell^{p_k}$, $x\mapsto (\phi_k(x)_1,\phi_k(x)_2)$ for $x \in \ell^{p_k}$ is a fixed linear  isomorphism
for each $k \in \mathbb N$.

We define  
\[
V_k=J_k\theta_k^{-1}\begin{bmatrix}
S_k&0\\0&0
\end{bmatrix}\theta_k P_k\in \mathcal S(X)
\]
for  $k \in \mathbb N$, see the diagram.
 \[
\xymatrix{X \ar[rr]^{V_k} \ar[d]_{P_k}  && X\\
X_k \ar[d]_{\theta_k} & & X_k \ar[u]_{J_k}\\
X_k\oplus X_k \ar[rr]^{\begin{bmatrix}
S_k&0\\0&0
\end{bmatrix}} && X_k\oplus X_k\ar[u]_{\theta_{k}^{-1}}
}
\]
Let 
\begin{equation}\label{092325}
\mathcal  A = \mathcal  A(\{V_k: k \in \mathbb N\})
\end{equation}
be the subalgebra of $\mathcal S(X)$ generated by the set $\{V_k: k \in \mathbb N\}$.
Finally, for each $n\in\mathbb N_0$ we put 
\begin{equation}\label{chain}
\mathcal  M_n
=\overline{\mathcal  A +  \mathcal S(X)^{n+1}}.
\end{equation}

We first verify that 
\[\mathcal  K(X)\subset\cdots\subset \mathcal   M_{n+1}\subset \mathcal  M_n\subset\cdots\subset \mathcal  M_0 = \mathcal  S(X),
\]
where  $\mathcal  M_{n}$ is a closed ideal of $\mathcal  M_{n-1}$ for  $n \in \mathbb N$. 
These facts will imply  that $\mathcal  M_n \in \mathcal{SI}_{n+1}(\mathcal L(X))$
for $n \in \mathbb N$.

Clearly $\mathcal  M_0 =\overline{\mathcal  A +  \mathcal S(X)} = \mathcal S(X)$, and $\mathcal  K(X)\subset \mathcal   M_n$ for $n \in \mathbb N$, since
$\mathcal S(X)^{n+1}$ is a non-zero ideal of $\mathcal L(X)$ and $X$ has the A.P. 
The inclusion $\mathcal S(X)^{n+1} \subset \mathcal S(X)^n$ yields that 
$\mathcal  M_{n} \subset \mathcal  M_{n-1}$ for $n \in \mathbb N$.
Suppose next that 
\[
U = A + T_1 \in \mathcal A + \mathcal S(X)^{n+1}, \quad V = B + T_2 \in \mathcal A + \mathcal S(X)^n
\]
are arbitrary,  with $A, B \in \mathcal A$, $T_1 \in \mathcal S(X)^{n+1}$ and $T_2 \in \mathcal S(X)^n$. Hence 
\[
UV = AB + AT_2 + T_1B + T_1T_2 \in \mathcal A + \mathcal S(X)^{n+1} \subset \mathcal  M_{n},
\]
since $AB \in \mathcal A$ and $AT_2 + T_1B + T_1T_2 \in \mathcal S(X)^{n+1}$. Similarly the product 
$VU \in  \mathcal A + \mathcal S(X)^{n+1}$. Thus we obtain  by approximation and  the continuity of the product  that
 $\mathcal  M_{n}$ is a closed ideal of $\mathcal  M_{n-1}$.
 
 \smallskip
 
 The main step of the argument consists of establishing the following 
 
 \smallskip
 
 \textit{Claim}. \textit{$\mathcal  M_n$ is not a closed  $n$-subideal of $\mathcal L(X)$}.
 
 \smallskip

Assume towards a contradiction that $\mathcal M_n\in\mathcal{SI}_n(\mathcal L(X))$. By Lemma \ref{newlemma} we have
\[
J_n\theta_n^{-1}\begin{bmatrix}
0&0\\S_n^n&0
\end{bmatrix}\theta_n P_n=\big(J_n\theta_n^{-1}\begin{bmatrix}
0&0\\ I_{X_n}&0
\end{bmatrix}\theta_n P_n \big)V_n^n\in \mathcal M_n
\]
as $V_n\in \mathcal M_n$. Above we applied  the identity 
$V_n^n = J_n\theta^{-1}_n\begin{bmatrix}
S_n^n&0\\0&0
\end{bmatrix}\theta_n P_n$.

Hence there are sequences $(T_r)\subset \mathcal A$ and $(W_r)\subset \mathcal S(X)^{n+1}$ such that 
\[
T_r+W_r\to J_n\theta_n^{-1}\begin{bmatrix}
0&0\\S_n^n&0
\end{bmatrix}\theta_n P_n
\] 
in $\mathcal S(X)$ as $r\to\infty$. 
It follows that 
\begin{equation}\label{contrad}
\theta_n P_n T_r J_n\theta_n^{-1}+\theta_n P_n W_r J_n\theta_n^{-1}\to \begin{bmatrix}
0&0\\S_n^n&0
\end{bmatrix}\end{equation}
in $\mathcal S(X_n \oplus X_n)$  as $r\to\infty$.

We will  establish that for each $n \in \mathbb N$  and $T\in \mathcal A$  there is an operator $U=U(T,n)\in \mathcal S(X_n)$ such that
\begin{equation}\label{claim1}
\theta_n P_n T J_n\theta_n^{-1}=\begin{bmatrix}
U&0\\0&0
\end{bmatrix}
\end{equation}
and that
\begin{equation}\label{claim2}
P_n A J_n\in \mathcal K(X_n)
\end{equation}
holds for all $A\in \mathcal S(X)^{n+1}$. Conditions \eqref{claim1} and \eqref{claim2} will then imply that 
\eqref{contrad} takes  the form
\[\begin{bmatrix}
U_r&0\\0&0
\end{bmatrix}+K_r\to\begin{bmatrix}
0&0\\S_n^n&0
\end{bmatrix} \textrm{ as } r \to \infty,
\]
where $U_r \in \mathcal S(X_n)$ and $K_r = \theta_n P_n W_r J_n\theta_n^{-1} \in \mathcal K(X_n\oplus X_n)$ for $r \in \mathbb N$.
In particular, above the operator norm convergence  as $r\to\infty$  of  the respective $(2,1)$-components  
yields that $S_n^n\in \mathcal K(X_n)$, which  contradicts  our original choice of $S_n$.
In the rest of the argument  we complete the proof  by showing that \eqref{claim1} and \eqref{claim2} hold.

\bigskip

\textit{Proof of \eqref{claim1}}. 

By definition \eqref{092325} and  linearity it is enough verify \eqref{claim1} for operators $T\in\mathcal A$ which have the form $T=V_{n_1}\cdots V_{n_k}$ where $n_1,\ldots, n_k\in\mathbb N$. 

Note that the inverses of $\theta_n$ satisfy
\[
\theta_n^{-1}:(x_1,\ldots, x_{n+1})+(y_1,\ldots,y_{n+1})\mapsto (\phi_1^{-1}(x_1,y_1),\ldots,\phi_{n+1}^{-1}(x_{n+1},y_{n+1}))
\]
for $(x_1,\ldots,x_{n+1}), (y_1,\ldots,y_{n+1}) \in X_n$.
For $t\ge s$ we let  $J_{t,s}:X_s\to X_t$ and $P_{s,t}:X_t\to X_s$ be the natural inclusion and projection maps. The following identities are straightforward to check:
\begin{align*}
J_s\theta_s^{-1} &=J_t\theta_t^{-1}(J_{t,s}\oplus J_{t,s}),\\
\theta_s P_s& =(P_{s,t}\oplus P_{s,t})\theta_t P_t
\end{align*}
and 
\[
(J_{t,s}\oplus J_{t,s})\begin{bmatrix}
S_s&0\\0&0
\end{bmatrix}(P_{s,t}\oplus P_{s,t})
 =\begin{bmatrix}
J_{t,s}S_sP_{s,t}&0\\0&0
\end{bmatrix}.
\]
Here $(P_{s,t}\oplus P_{s,t})(x,y) = (P_{s,t}x,P_{s,t}y)$ and $J_{t,s}\oplus J_{t,s}$ is defined in a similar manner.
By using these identities we obtain that
\begin{align*}
V_s&=J_s\theta^{-1}_s 
\begin{bmatrix}
S_s&0\\0&0
\end{bmatrix}\theta_s P_s\\
&=J_t\theta_t^{-1}(J_{t,s}\oplus J_{t,s})\begin{bmatrix}
S_s&0\\0&0
\end{bmatrix}(P_{s,t}\oplus P_{s,t})\theta_t P_t
\\
&=J_t\theta_t^{-1}\begin{bmatrix}
J_{t,s}S_sP_{s,t}&0\\0&0
\end{bmatrix}\theta_t P_t,
\end{align*}
where $J_{t,s}S_sP_{s,t} \in \mathcal S(X_t)$.

It follows that if $t\ge max\{r,s\}$, then both the products $V_r V_s$ and $V_sV_r$ have the form
\begin{equation}\label{te}
J_t\theta_t^{-1}\begin{bmatrix}
U&0\\0&0
\end{bmatrix}\theta_t P_t\end{equation}
for some $U\in \mathcal S(X_t)$. By iteration, any finite product $V_{n_1}\cdots V_{n_k}$ has the form \eqref{te} if $t\ge max\{n_1,\ldots, n_k\}$. 

Suppose next that $T=V_{n_1}\cdots V_{n_k}$,  where $n_1,\ldots ,n_k \in \mathbb N$ are arbitrary. If $max\{n_1,\ldots,n_k\}\le n$, then  according to \eqref{te}
we have 
\[
T=J_n\theta_n^{-1}\begin{bmatrix}
U&0\\0&0
\end{bmatrix}\theta_{n} P_n
\]
for some $U\in \mathcal S(X_n)$, so 
$\theta_n P_n T J_n \theta_n^{-1}=\begin{bmatrix}
U&0\\0&0
\end{bmatrix}$ 
has the desired form \eqref{claim1}. 

On the other hand, if $s:=max\{n_1,\ldots,n_k\}>n$, then 
\[
T=J_s\theta_s^{-1}\begin{bmatrix}
U&0\\0&0
\end{bmatrix}\theta_s P_s\]
for some $U\in \mathcal S(X_s)$ by \eqref{te}. By using the simple identities
\[
\theta_s P_sJ_n\theta_n^{-1}=J_{s,n}\oplus J_{s,n}, \quad  \theta_n P_n J_s\theta_s^{-1}=P_{n,s}\oplus P_{n,s},
\]
we get that
\begin{align*}
\theta_n P_n T J_n\theta_n^{-1} & =\theta_n P_n (J_s\theta_s^{-1}\begin{bmatrix}
U&0\\0&0
\end{bmatrix}\theta_s P_s) J_n\theta_n^{-1}\\ 
& =(P_{n,s}\oplus P_{n,s})\begin{bmatrix}
U&0\\0&0
\end{bmatrix}(J_{s,n}\oplus J_{s,n}) \\
& = \begin{bmatrix}
P_{n,s}UJ_{s,n}&0\\0&0
\end{bmatrix}.
\end{align*}
Thus we also obtain the desired identity  \eqref{claim1} in the second  case since $P_{n,s}UJ_{s,n}\in \mathcal S(X_n)$.

\smallskip

\textit{Proof of \eqref{claim2}}.  We claim that   $P_n A J_n\in \mathcal K(X_n)$ for all $A\in \mathcal S(X)^{n+1}$.

\smallskip

By linearity we may assume that $A=A_1\cdots A_{n+1}$, where $A_i\in \mathcal S(X)$ for $i \in \{1,\ldots,n+1\}$.
Let $Q_k:X\to \ell^{p_k}$ and $q_k:X_n\to \ell^{p_k}$ be the natural projections, 
and $R_k:\ell^{p_k}\to X$ and $r_k:\ell^{p_k}\to X_n$ be the natural inclusion maps for $1 \le k \le n+1$. 
Thus $P_n=r_1Q_1+\ldots+r_{n+1}Q_{n+1}$ and $J_n=R_1q_1+\ldots +R_{n+1}q_{n+1}$,  so that
\[
P_n A J_n=\sum_{s,t =1}^{n+1}  r_t Q_t A R_s q_s .
\]
In order to verify \eqref{claim2} it will suffice to show that 
\begin{equation}\label{enoughnew}
 Q_t A R_s=Q_t A_1\cdots A_{n+1} R_s\in \mathcal K(\ell^{p_s},\ell^{p_t})
\end{equation} 
holds for all $1\le s, t \le n+1$. Note that \eqref{enoughnew} holds for $s\ge t$ by \eqref{prop1} and \eqref{prop2}, and thus we may assume that $1\le s<t\le n+1$. 

We will apply the following result from  \cite[Corollary 5.2]{L01} in order to verify \eqref{enoughnew}: 
If $U,V\in \mathcal  L(X)$ and $s, t \in\mathbb N$ then
\begin{equation}\label{matrixmult.new}
Q_tUVR_s =\sum_{k=1}^\infty Q_tUR_kQ_kVR_s\ , 
\end{equation}
where the series converges in the operator norm of $\mathcal L(\ell^{p_s},\ell^{p_t})$. 
Note that \cite[Corollary 5.2]{L01} (see also condition (ii) on \cite[page 166]{L01}) applies here since
$\mathcal  L(E,\ell^{p_j})= \mathcal  K(E,\ell^{p_j})$ for all $j\in\mathbb N$ by Pitt's theorem \eqref{prop2} for $E = c_0$ or $E = \ell^r$ with $r \ge \sup_j p_j$. 

Let $U_m=\sum_{k=1}^m R_kQ_k$ for all $m\in\mathbb N$, and define for all $r=1,\ldots,n$ and $m_1,\ldots,m_r\in\mathbb N$ the auxiliary operators
\[B_{(m_1,\ldots,m_r)}:=Q_t A_1U_{m_1}\cdots A_rU_{m_r}A_{r+1}A_{r+2}\cdots A_{n+1}R_s.\]
We first claim that for all $\varepsilon>0$ there are $m_1,\ldots, m_n\in\mathbb N$ such that
\begin{equation}\label{0922252}
||B_{(m_1,\ldots,m_n)}-T||<\varepsilon
\end{equation}
where $T=Q_t A_1\cdots A_{n+1}R_s$ is the operator in \eqref{enoughnew}. 

To see this let $\varepsilon>0$. By \eqref{matrixmult.new} (with $U=A_1$ and $V=A_2\cdots A_{n+1}$) there is $m_1\in\mathbb N$ such that $||B_{(m_1)}-T||<\varepsilon/n$. If $n\ge 2$, we continue by picking $m_2\in\mathbb N$ such that \[||B_{(m_1)}-B_{(m_1,m_2)}||<\varepsilon/n\] (here we again use \eqref{matrixmult.new} with $U=A_1 U_{m_1}A_2$ and $V=A_3\cdots A_{n+1}$). Assume that $2\le r\le n-1$ and we have picked $m_1,\ldots,m_r\in\mathbb N$ such that \[||B_{(m_1,\ldots,m_{k-1})}-B_{(m_1,\ldots,m_k)}||<\varepsilon/n\] 
for $k=2,\ldots,r$.
Next we use \eqref{matrixmult.new} (with $U=A_1U_{m_1}A_2U_{m_2}\cdots A_rU_{m_r}A_{r+1}$ and $V=A_{r+2}\cdots A_{n+1}$) to pick $m_{r+1}\in\mathbb N$ such that
\[||B_{(m_1,\ldots,m_r)}-B_{(m_1,\ldots,m_{r+1})}||<\varepsilon/n.\]
Hence we get \eqref{0922252} from the triangle inequality.

Next we note that
\begin{align*}
B_{(m_1,\ldots,m_n)}&=Q_t A_1U_{m_1}A_2U_{m_2}\cdots A_nU_{m_n} A_{n+1}R_s\\
&=\sum_{k_1=1}^{m_1}\cdots\sum_{k_n=1}^{m_n}Q_t A_1R_{k_1}Q_{k_1}A_2R_{k_2}Q_{k_2}\cdots A_nR_{k_n}Q_{k_n}A_{n+1}R_s.
\end{align*}
In order to show \eqref{enoughnew} it will be enough in view of \eqref{0922252} to verify that
\begin{equation}\label{2109}
Q_t A_1 R_{k_1}Q_{k_1}A_2 R_{k_2}Q_{k_2}\cdots A_n R_{k_n}Q_{k_n}A_{n+1} R_s\in\mathcal K(\ell^{p_s},\ell^{p_t})\end{equation}
holds for all  $k_1,\ldots,k_n\in\mathbb N$. 

If $k_r\ge t$ or $k_r\le s$ for some $r=1,\ldots,n$, then \eqref{prop2} - \eqref{prop3}  imply that \eqref{2109} holds. In fact, if $k_r\ge t$, then 
\[Q_tA_1R_{k_1}Q_{k_1}\cdots A_rR_{k_r}\in\mathcal K(\ell^{p_{k_r}},\ell^{p_t})\]
and thus \eqref{2109} holds. The case $k_r\le s$ follows similarly by observing that
\[Q_{k_r}A_{r}R_{k_{r}}Q_{k_{r}}\cdots A_{n+1}R_s\in\mathcal K(\ell^{p_s},\ell^{p_{k_r}}).\]
 On the other hand, if $s<k_r<t$ for all $r=1,\ldots,n$, we must have $k_l=k_m$ for some $1\le l<m\le n$ since $\#(\{s+1,\ldots,t-1\})<n$ (recall that $1\le s<t\le n+1$). By \eqref{prop3} we have
\[Q_{k_l}A_{l+1}R_{k_{l+1}}Q_{k_{l+1}}\cdots A_{m}R_{k_m}\in\mathcal K(\ell^{p_{k_l}})\]
which again implies \eqref{2109}.

As pointed out above this establishes the second claim \eqref{claim2}. 
Altogether we have completed the argument that 
$\mathcal M_n \notin \mathcal{SI}_{n}(\mathcal L(X))$, and consequently  that of Theorem \ref{nsubidchain}.
\end{proof}

\begin{remark}
The result of Theorem \ref{nsubidchain} also holds for the  direct $\ell^r$-sum 
\[
Y = \big(\bigoplus_{j=1}^\infty \ell^{p_j}\big)_{\ell^r} ,
\]
where $(p_j)$ is any strictly increasing sequence in $(1,\infty)$ and $1 \le r < p_1$.
In this case condition (i) from  \cite[page 166]{L01}  is satisfied by 
\eqref{prop2}, so that  the convergence formula \eqref{matrixmult.new} from \cite[Corollary 5.2]{L01} is also valid in $\mathcal L(Y)$.
\end{remark}
Next we return to the setting of Theorem \ref{general} and include a variant of Theorem \ref{nsubids2} for finite sums of different $p$-James spaces. The underlying spaces are quasi-reflexive 
and the reference ideals $\mathcal I$ in the application of 
Theorem \ref{general} can  be chosen differently from Theorem \ref{nsubids2}. (See  Remarks \ref{jamesvers} for further comments.)

Let $1 < p < \infty$. The $p$-James space  $J_p$ consists of the scalar-valued sequences $x = (x_j)$ for which $\lim_{j\to\infty} x_j = 0$ and 
the $p$-variation norm 
\begin{equation}\label{jp}
\Vert x \Vert_{J_{p}} = \sup \{ \big(\sum_{k=1}^n \vert x_{j_{k+1}} - x_{j_{k}}\vert^p\big)^{1/p}: 1 \le j_1 < \cdots < j_{n+1}, n \in \mathbb N\}
\end{equation}
is finite.  Here $J_p^{**}$ is identified with the space of scalar sequences $x = (x_j)$ for which $\Vert x \Vert_{J_{p}} < \infty$. In fact, $J_p^{**} = J_p \oplus \mathbb K e$, 
where  $e = (1,1,1,\dots) \in J_p^{**}$.

Let  $n \in \mathbb N$ and $1 <  p_1<p_2<\cdots<p_{n+1} <  \infty$. We will consider the max-normed direct sums  
\begin{equation}\label{elem2}
Y_n = \bigoplus_{j=1}^{n+1} J_{p_{j}}. 
\end{equation}

Let $\mathcal W(X,Y)$ denote the class of weakly compact operators $X \to Y$ for Banach spaces
$X$ and $Y$. The class  $\mathcal W$ is a closed Banach operator ideal, so that $\mathcal W(X)$ is a closed ideal of $\mathcal L(X)$ for any $X$.

We  will require the following properties of the spaces $J_p$ and their bounded operators:  
\smallskip

(i) For $1<p<\infty$ we have $\mathcal L(J_p) =  span(I_{J_p}) + \mathcal W(J_p)$ and
\begin{equation}\label{james1}
 \mathcal S(J_p) = \mathcal K(J_p) \varsubsetneq \mathcal W(J_p),
\end{equation}
see \cite[p. 344]{LW89} and \cite[Proposition 4.9]{L02}. Moreover, there is a complemented subspace $R_p \subset J_p$ such that
\begin{equation}\label{james5}
R_p \approx \ell^p,
\end{equation}
see e.g.  \cite[Corollary 4.7]{L02}.
\smallskip

(ii) For $1 < p < q < \infty$  the inclusion map $i_{p,q}: J_p \to J_q$ is not weakly compact and 
\begin{equation}\label{james2}
\mathcal L(J_p,J_q)= span(i_{p,q}) +  \mathcal W(J_p,J_q),
\end{equation}
see  \cite[p. 344]{LW89}. In addition, 
\begin{equation}\label{james4}
\mathcal L(J_p,J_q) = \mathcal S(J_p,J_q)
\end{equation} since the spaces $J_p$ and $J_q$ are totally incomparable, which follows from \cite[Corollary 4.7]{L02}.
\smallskip

(iii) For $1 < q < p < \infty$ we have 
\begin{equation}\label{james3}
\mathcal L(J_p,J_q) = \mathcal K(J_p,J_q),
\end{equation}
see \cite[Theorem 4.5]{LW89}. 

Let $Z$ be a Banach space. In the following results we use the standard notation $\mathcal I(Z)$ for the closed ideal of $\mathcal L(Z)$ consisting of the operators $Z\to Z$ which belong to the closed Banach operator ideal $\mathcal I=\mathcal S\cap\mathcal W$ or $\mathcal I=\mathcal S$.

 \begin{theorem}\label{nsubids3}
Let  $n \in \mathbb N$ and $Y_n$  be the direct sum space from \eqref{elem2}, where $1 <  p_1< \cdots<p_{n+1} <   \infty$. If
$\mathcal I(Y_n) =  \mathcal S(Y_n) \cap \mathcal W(Y_n)$ or $\mathcal I(Y_n) =  \mathcal S(Y_n)$, then
\begin{enumerate}
\item[(i)]  the quotient algebra $\mathcal I(Y_n)/ \mathcal K(Y_n)$ is $(n+1)$-nilpotent, and 
\[
\overline{\mathcal I(Y_n)^{n+1}} = \mathcal K(Y_n),
\]

\smallskip

\item[(ii)] there is  $T \in \mathcal I(Y_n)$ for which $T^n  \notin \mathcal K(Y_n)$, 

\smallskip

\item[(iii)]  for $Z_n = Y_n \oplus Y_n$ there is a closed subalgebra  $\mathcal M \subset \mathcal I(Z_n)$
such that 
\[
\mathcal M \in \mathcal{SI}_{n+1}(\mathcal L(Z_n)) \setminus \mathcal{SI}_{n}(\mathcal L(Z_n)).
\]
\end{enumerate}
\end{theorem} 

\begin{proof}  The argument follows the outline from Theorem \ref{nsubids2}.  Let $P_k$ and $I_k$ denote the canonical projection and  inclusion 
related to the component $J_{p_k}$ of $Y_n$ for $k=1,\ldots,n+1$.

(i) This is a  modification of the pigeon-hole argument used in part (i) of Theorem \ref{nsubids2}. One requires the additional facts that
\[
\mathcal S(J_{p_r}) = \mathcal S(J_{p_r}) \cap \mathcal W(J_{p_r}) = \mathcal K(J_{p_r}) 
\]
for $1 \le r \le n+1$  in view of  \eqref{james1}, where $\mathcal S \cap \mathcal W$ is an operator ideal. Recall further that the spaces 
$J_p$ have a Schauder basis, so $Y_n$ has the A.P.

\smallskip

(ii)  For $\mathcal I(Y_n)=\mathcal S(Y_n) \cap \mathcal W(Y_n)$ we use \eqref{james5} to pick a complemented subspace $R_{p_k} \subset J_{p_k}$ such that 
$R_{p_k} \approx \ell^{p_k}$ for  $k = 1, \ldots, n+1$. Fix a linear isomorphism $U_k: R_{p_k} \to \ell^{p_k}$, and
let $q_k: J_{p_k} \to R_{p_k}$ be a projection and $j_k$ the inclusion map $R_{p_k} \to  J_{p_k}$, so that $q_kj_k = I_{R_{p_{k}}}$
 for $k = 1,\ldots, n+1$. We define 
 \[
 r_k = j_{k+1}U_{k+1}^{-1}i_kU_kq_k \in \mathcal L(J_{p_k},J_{p_{k+1}}),
 \]
 where $i_k: \ell^{p_k} \to \ell^{p_{k+1}}$ is the inclusion map  for $k = 1,\ldots,n$. 
 \[
\xymatrix{J_{p_k} \ar[rr]^{r_k} \ar[d]_{q_k}  && J_{p_{k+1}}\\
R_{p_k} \ar[d]_{U_k} & & R_{p_{k+1}} \ar[u]_{j_{k+1}}\\
\ell^{p_k}  \ar[rr]^{i_k} && \ell^{p_{k+1}}\ar[u]_{U_{k+1}^{-1}}
}
\]
Recall that
$r_k \in  \mathcal S(J_{p_k},J_{p_{k+1}}) \cap \mathcal W(J_{p_k},J_{p_{k+1}})$ 
by the reflexivity of $\ell^{p_k}$ and the total incomparability of $\ell^p$ and $\ell^q$ for $p \neq q$.
Hence 
 \[
T=\sum_{k=1}^{n}I_{k+1}r_kP_k \in \mathcal S(Y_n) \cap \mathcal W(Y_n).
\] 
By iteration we get that 
\[
T^nx = (0,\ldots,0,r_n \cdots r_1x_1),\quad x = (x_1,\ldots,x_{n+1}) \in Y_n.
\]
It follows that  $P_{n+1}T^nI_1 = r_n \cdots  r_1 \notin \mathcal K(J_{p_1},J_{p_{n+1}})$, 
as the restriction  
\[
(r_n \cdots r_1)_{|R_{p_{1}}} = j_{n+1}U_{n+1}^{-1}(i_n\cdots i_1)U_1
\] 
is up to linear isomorphisms the inclusion map $\ell^{p_1} \to \ell^{p_{n+1}}$.
Conclude that $T^n \notin \mathcal K(Y_n)$. 

Clearly $T \in \mathcal S(Y_n) \cap \mathcal W(Y_n)$ also works for $\mathcal I(Y_n)=\mathcal S(Y_n)$.
(The choice  
\[
S=\sum_{k=1}^{n}I_{k+1}i_{p_k,p_{k+1}}P_k \in \mathcal S(Y_n)
\] 
is a simpler one,  since the inclusion maps $i_{p_k,p_{k+1}}:  J_{p_k}\to J_{p_{k+1}}$ are 
strictly singular for $k=1,\ldots,n$ and $i_{p,q} \notin \mathcal W(J_p,J_q)$ for $q>p$ by \eqref{james2}.)

\smallskip

Finally,  (iii)  follows immediately  from the proof of Theorem \ref{general} and parts (i) - (ii).
\end{proof}

\begin{remarks}\label{jamesvers}
(1) Properties \eqref{james1},  \eqref{james4} and \eqref{james3} yield  that  
\begin{align*}
\mathcal S(Y_n)  =  \{[U_{jk}]  \in \mathcal L(Y_n): &\  U_{jk} \textrm{ is compact for } 1 \le j \le  k  \le n+1,\\
&  U_{jk} \textrm{ is bounded for } 1 \le  k < j   \le n+1\},
\end{align*}
\begin{align*}
\mathcal S(Y_n) \cap \mathcal W(Y_n)  =   \{[U_{jk}]  &\in \mathcal L(Y_n):  U_{jk} \textrm{ is compact for } 1 \le j \le  k  \le n+1,\\
&  U_{jk} \textrm{ is weakly compact for } 1 \le  k < j   \le n+1\}.
\end{align*}
It is known  that $\mathcal S(Y_n) \cap\mathcal W(Y_n)\subsetneq \mathcal S(Y_n)$, so the two choices for $\mathcal I$ in Theorem \ref{nsubids3} are different. In fact,
the inclusion maps $i_{p_r,p_s}:  J_{p_r}\to J_{p_s}$  extend to strictly singular non-weakly compact operators on $Y_n$  for $1\le r<s\le n+1$ in view of  \eqref{james2}  and \eqref{james4}. 
\smallskip

(2) Theorem \ref{general} cannot be applied to the closed ideal $\mathcal I(Y_n)=\mathcal W(Y_n)$, since $\mathcal W(Y_n)=\mathcal W(Y_n)^k$ for all $k\ge 2$. The reason is that $\mathcal W(Y_n)$ has a bounded left approximate identity, that is, there is a bounded net $(V_\alpha)\subset \mathcal W(Y_n)$ such that 
\begin{equation}\label{092725}
\lim_\alpha V_\alpha V=V
\end{equation}
 for all $V\in\mathcal W(Y_n)$. Hence the Cohen-Hewitt factorisation theorem, see e.g. \cite[Corollary 2.9.26]{D00}, implies that $\mathcal W(Y_n)^2=\mathcal W(Y_n)$. The fact \eqref{092725} is seen by combining \cite[Theorem 2.2 and Proposition 2.5]{OT05} with the comment before \cite[Problem 2.4]{OT05} and \cite[Proposition 5.3]{OT05}. In particular, the choice of the closed reference ideal $\mathcal I$ is essential in Theorem \ref{general}.

\smallskip

(3)  It is not difficult to check that $dim(Y_n^{**}/Y_n) = n+1$, so $Y_n$ are quasi-reflexive spaces. This implies that $Y_n \oplus Y_n \not\approx Y_n$, see e.g. \cite{BP60}.
\end{remarks}

\smallskip

The  non-trivial closed $(n+1)$-subideals in Theorems \ref{nsubids2}, \ref{nsubidchain} and \ref{nsubids3}  
are defined on spaces $X$ which have the A.P., where 
$\mathcal K(X)$ is the smallest non-zero subideal. The earlier examples of non-trivial $2$-subideals from  \cite{TW23}, as well as  Sections \ref{Tarbard} and \ref{newsubs},
raise the question whether there are spaces $X$ (without the A.P.)
that allow non-trivial closed subideals $\mathcal J \in \mathcal{SI}_{n+1}(\mathcal L(X))$ for $n \ge 2$ such that $\mathcal A(X) \varsubsetneq \mathcal J \varsubsetneq  \mathcal K(X)$.
It turns out that this is possible, though more sophisticated tools and concepts are  involved.

Let $X$ and $Y$ be Banach spaces and $p \in [1,\infty)$. 
Recall from \cite{PP69} that $T \in \mathcal L(X,Y)$ is a
\textit{quasi $p$-nuclear} operator, denoted $T \in \mathcal{QN}_p(X,Y)$, if there is a strongly $p$-summable sequence $(x_k^*)$ in $X^*$ such that 
\begin{equation}\label{qnp}
\Vert Tx\Vert \le \Big(\sum_{k=1}^\infty \vert x_k^*(x)\vert^p\Big)^{1/p}, \quad x \in X.
\end{equation}
It is known \cite[Section 4]{PP69} that $(\mathcal{QN}_p, \vert \cdot \vert_p)$ defines a Banach operator ideal, where the 
complete ideal  norm is given by 
\[
\vert T\vert_p = \inf\Big\{ \Big(\sum_{k=1}^\infty \Vert x_k^*\Vert^p\Big)^{1/p}:   (x_k^*)  \textrm{ satisfies  } \eqref{qnp}\Big\}.
\]

Let $Z$ be any Banach space. We will require the monotonicity and the H\"{o}lder properties of the classes  $\mathcal{QN}_p$, that is, 
\begin{equation}\label{mono} 
\mathcal{QN}_p(Z)\subset\mathcal{QN}_q(Z) \textrm{ and } \Vert \cdot \Vert \le \vert \cdot \vert_q \le \vert \cdot \vert_p
\end{equation}
 for $p<q$ by \cite[Satz 21 and 24]{PP69}, and for $1/r=1/p+1/q \le 1$ one has 
\begin{equation}\label{hol} 
\textrm{ if } S\in\mathcal{QN}_p(Z)  \textrm{ and } T\in\mathcal{QN}_q(Z),  \textrm{ then } ST\in \mathcal{QN}_{r}(Z), 
 \end{equation}
 see \cite[Satz 48]{PP69}. Moreover, $\mathcal{QN}_p(Z) \subset \mathcal K(Z)$ for any $p$ by \cite[Satz 25]{PP69}, and 
 \begin{equation}\label{qnapp}
 \mathcal{QN}_2(Z) \subset \mathcal A(Z),
  \end{equation}
 see \cite[Remarks 3.2.(ii)]{W23} for references to the ingredients of the argument.
  
The following theorem contains another central result of this section.  
It makes use of a closed subspace of $c_0$ and an associated quasi $p$-nuclear operator from \cite[Section 4]{W23}, 
which in turn depend on constructions of Davie \cite{D75} of Banach spaces without the 
A.P.  and a factorization result of Reinov \cite[Lemma 1.1]{R82}. 
Part (i) of  the theorem produces for $n = 1$ a closed subspace $Z_1 \subset c_0$ and a non-trivial closed $2$-subideal $\mathcal I$ of $\mathcal L(Z_1)$, which 
differs from the examples in \cite{TW23} and Section 2. In part (i) the closed subspaces $Z_n \subset c_0$ and 
$\mathcal M_{n} \in \mathcal{SI}_{n+1}(\mathcal L(Z_n))$ 
depend on $n \in \mathbb N$, but in (ii) we lift them from $\mathcal L(Z_n)$ to $\mathcal L(Z)$ for the direct sum $Z = \big(\bigoplus_{n = 1}^{\infty} Z_n\big)_{c_0}$.

 \begin{theorem}\label{nsubid}
 (i)  Fix $n \in \mathbb N$, and assume that  $2n <  p \le 2(n+1)$. 
  Then  there is a closed subspace $Z_n \subset c_0$ without the A.P.  and a closed subalgebra  $\mathcal M_{n}$ of $\overline{\mathcal{QN}_p(Z_n)}$ such that 
  \[
  \mathcal M_{n} \in \mathcal{SI}_{n+1}(\mathcal L(Z_n)) \setminus \mathcal{SI}_{n}(\mathcal L(Z_n)).
  \]

(ii) Let $Z = \big(\bigoplus_{n = 1}^{\infty} Z_n\big)_{c_0}$. Then  
 \[
 \mathcal{SI}_{n}(\mathcal L(Z)) \varsubsetneq  \mathcal{SI}_{n+1}(\mathcal L(Z)) 
 \]
 for all $n \in \mathbb N$.
 \end{theorem}
 
 \begin{proof}
(i) Fix $n \in \mathbb N$ and let  $p \in (2n, 2(n+1)]$. According to \cite[Proposition 4.3]{W23} there is a closed subspace $X_n \subset c_0$ together with 
$S \in\mathcal{QN}_p(X_n)$ such that $S^n\notin \mathcal A(X_n)$. 
By construction  $X_n$ fails the  A.P.  

In order to apply Theorem \ref{general} we require the following observation, which is essentially contained in \cite[Proposition 4.1 and Remarks 4.2.(ii)]{W23}:
if $2n <  p \le 2(n+1)$, then
\begin{equation}\label{powers1} 
\overline{\mathcal{QN}_p(Y)^{n+1}} = \mathcal A(Y)
\end{equation}
holds for any Banach space $Y$.
In fact, the H\"older property \eqref{hol} together with  \eqref{mono} and \eqref{qnapp} imply  that
\begin{equation*}\label{qnhol1}
\mathcal{QN}_p(Y)^{n+1}  \subset \mathcal{QN}_{p/(n+1)}(Y) \subset \mathcal{QN}_2(Y)  \subset \mathcal A(Y),
\end{equation*}
as $p/(n+1) \le 2$ by assumption. Since $\mathcal F(Y) \subset \mathcal{QN}_p(Y)^{n+1}$ we obtain 
\eqref{powers1} by passing to the closures. 

Hence the preceding operator  $S \in\mathcal{QN}_p(X_n)$  satisfies $S^n \notin \overline{\mathcal{QN}_p(X_n)^{n+1}}$  by  \eqref{powers1}.  
Let $Z_n=X_n\oplus X_n \subset c_0 \oplus c_0 \simeq  c_0$ (isometric isomorphism). Thus $Z_n \subset c_0$ is a closed subspace that  fails the A.P.  We put
\[
 T=
 \begin{bmatrix}
 S&0\\
0&0
\end{bmatrix}\in\mathcal{QN}_p(Z_n)
\]
and define as in  Theorem \ref{general}
\begin{equation}\label{lastideal2}
\mathcal M_{n}= span(T,\ldots,T^{n})+ \overline{\mathcal{QN}_p(Z_n)^{n+1}}. 
\end{equation}

We may now apply Theorem \ref{general} for the closed ideal $\mathcal J = \overline{\mathcal{QN}_p(Z_n)}$, and deduce that the closed subalgebra 
\[
\mathcal M_{n} \in \mathcal{SI}_{n+1}(\mathcal L(Z_n)) \setminus \mathcal{SI}_{n}(\mathcal L(Z_n)).
\]

\smallskip

(ii)  We again use $J_n: Z_n \to Z$ and $P_n: Z \to Z_n$ for 
the canonical inclusions, respectively the canonical projections associated to $Z$. 

For all $n\in\mathbb N$ pick $p_n\in(2n,2(n+1)]$ and let $T_n =
 \begin{bmatrix}
 S_n&0\\
0&0
\end{bmatrix} \in \mathcal{QN}_{p_{n}}(Z_n)$ be the operator from  the proof of part (i).
 (Note that the operators $T_n$ and $S_n$ as well as $p_n$ depend on $n$  in that argument, 
 even though the  dependence was not displayed explicitly.)  Consider
 \[
 R_n = J_nT_nP_n \in \mathcal{QN}_{p_{n}}(Z)
 \]
for all $n\in\mathbb N$  and put 
 \[
\mathcal  K_{n+1} := span(R_n,\ldots, R_n^n) + \mathcal A(Z).
 \]
 
\textit{Claim}.  The closed subalgebra $\mathcal K_{n+1}$ of $\mathcal L(Z)$  satisfies 
 \[
\mathcal K_{n+1}  \in  \mathcal{SI}_{n+1}(\mathcal L(Z)) \setminus \mathcal{SI}_{n}(\mathcal L(Z)) 
 \]
for all $n\in\mathbb N$.
 
 To see this we introduce the intermediary closed linear subspaces 
 \[
\mathcal  L_k := span(R_n,\ldots, R_n^k) + \overline{\mathcal{QN}_{p_{n}}(Z)^{k+1}} \subset \overline{\mathcal{QN}_{p_{n}}(Z)}
 \]
 for $k = 1,\ldots, n$.  We know  that $\overline{\mathcal{QN}_{p_n}(Z)^{n+1}} = \mathcal A(Z)$ in view of  \eqref{powers1} 
 as $2n < p_n \le 2(n+1)$, so that $\mathcal  L_n = \mathcal K_{n+1}$. By arguing as in the proof of Theorem \ref{general} we obtain that
 the chain 
 \[
  \mathcal K_{n+1} = \mathcal  L_n \subset \mathcal  L_{n-1} \subset  \cdots \subset \mathcal  L_1 \subset \overline{\mathcal{QN}_{p_{n}}(Z)}
   \subset  \mathcal L(Z)
 \]
consists of closed relative ideals.  It follows that $\mathcal K_{n+1}$ is a closed $n$-subideal of $\overline{\mathcal{QN}_{p_{n}}(Z)}$, so that 
$\mathcal K_{n+1} \in  \mathcal{SI}_{n+1}(\mathcal L(Z))$. 

\smallskip

Suppose next to the contrary that $\mathcal K_{n+1} \in \mathcal{SI}_{n}(\mathcal L(Z))$. Hence, by Lemma \ref{newlemma}, we have
\[J_n  \begin{bmatrix}
0&0\\
S_n^n&0
\end{bmatrix} P_n  =
J_n  \left(\begin{bmatrix}
0&0\\
I_{X_{n}}&0
\end{bmatrix} T_n^n \right) P_n 
 = \left(J_n  \begin{bmatrix}
0&0\\
I_{X_{n}}&0
\end{bmatrix} P_n\right)R^n_n 
\in   \mathcal  K_{n+1} .\]
Above  we have applied the identities  $P_n J_n = I_{Z_{n}}$ and $R^{n}_n = J_nT^{n}_nP_n$.
Observe next that if 
\[
U = \sum_{k=1}^n a_k R_n^k + V \in \mathcal  K_{n+1} = span(R_n,\ldots, R_n^n) + \mathcal A(Z)
\]
for some scalars $a_1,\ldots , a_n$ and $V \in \mathcal A(Z)$, then
\[
P_nUJ_n =  \sum_{k=1}^n a_k P_n R_n^kJ_n + P_nVJ_n =   \sum_{k=1}^n a_k T_n^k + P_nVJ_n \in   \mathcal M_{n},
\]
where  $\mathcal M_{n} \subset \mathcal L(Z_n)$ is the closed $(n+1)$-subideal  defined in  \eqref{lastideal2}. In particular, 
$\begin{bmatrix}
0&0\\
S_n^n&0
\end{bmatrix}\in \mathcal M_{n}$, which was already shown to be impossible during the proof of Theorem \ref{general} (see the verification of \eqref{1409}). 
\end{proof}

The Banach operator ideal  component $\mathcal B := (\mathcal{QN}_p(X),\vert \cdot \vert_p)$ is also a Banach algebra in its own right, where 
$\vert \cdot \vert_p$ is the associated complete ideal norm. 
We include a variant  of part (i) of Theorem \ref{nsubid} for $\mathcal B$, where  the subalgebras and $n$-subideals 
are closed in the  $\vert \cdot \vert_p$-norm. This setting is somewhat outside of Theorem \ref{general} (but Lemma \ref{newlemma} applies).
 
 \begin{proposition}\label{nsubidqn}
  For any $n \ge 2$ and $p \in (2n,2(n+1)]$ there is a closed subspace $Z_n \subset c_0$ and a closed subalgebra  
  $\mathcal J_n \subset \mathcal B := (\mathcal{QN}_p(Z_n),\vert \cdot \vert_p)$ such that 
  \[
   \mathcal J_n \in \mathcal{SI}_{n}(\mathcal B) \setminus \mathcal{SI}_{n-1}(\mathcal B).
 \]
 \end{proposition}
 
 \begin{proof}
 We sketch the main differences to Theorems  \ref{general} and \ref{nsubid}.  
 
 Fix $n \ge 2$ and $p \in (2n,2(n+1)]$. As  in the proof of Theorem  \ref{nsubid}  we pick a closed linear subspace $X_n \subset c_0$ and 
an operator  $S \in \mathcal{QN}_p(X_n)$ such that  $S^n \notin \mathcal A(X_n)$. 
Let $Z_n=X_n\oplus X_n$ and consider 
$T=
 \begin{bmatrix}
 S&0\\
0&0
\end{bmatrix}\in\mathcal{QN}_p(Z_n)$. For $1 \le k \le n$ we define
\[
\mathcal J_k  :=span(T,\ldots,T^k)+\overline{\mathcal B^{k+1}}^{\vert \cdot \vert_p}.
\]
By imitating  the proof of Theorem \ref{general} one verifies that the chain 
\begin{equation}\label{newchain1}
\mathcal J_n \subset \mathcal J_{n-1} \subset \cdots \subset  \mathcal J_1 \subset \mathcal B = \mathcal{QN}_p(Z_n)
\end{equation}
consists of $\vert \cdot \vert_p$-closed relative ideals. Recall here that $\vert \cdot \vert_p$ is a Banach algebra norm, so that multiplication is continuous on $\mathcal B$. 
The explicit chain \eqref{newchain1} implies that  $\mathcal J_n \in \mathcal{SI}_{n}(\mathcal B)$.

\smallskip

Next, suppose to the contrary that $\mathcal J_n\in\mathcal{SI}_{n-1}(\mathcal B)$.
By Lemma \ref{newlemma} we have 
 \begin{equation}\label{iterate}
\begin{bmatrix}
0&0\\
S^n&0
\end{bmatrix} = 
\begin{bmatrix}
0&0\\
S&0
\end{bmatrix} \circ  T^{n-1} \in  \mathcal J_n,
\end{equation}
since $\begin{bmatrix}
0&0\\
S&0
\end{bmatrix} \in\mathcal B$.  
We may thus write 
\[
\begin{bmatrix}
0&0\\
S^n&0
\end{bmatrix} = \sum_{k=1}^n a_kT^k + U = \begin{bmatrix}
\sum_{k=1}^n a_kS^k+U_{11}&U_{12}\\
U_{21}&U_{22}
\end{bmatrix},
\]
where  $U = \begin{bmatrix}
U_{11}&U_{12}\\
U_{21}&U_{22}
\end{bmatrix} \in  \overline{\mathcal B^{n+1}}^{\vert \cdot \vert_p}$. We observe next that  
\begin{equation}\label{qnp1}
\overline{\mathcal B^{n+1}}^{\vert \cdot \vert_p} \subset  \mathcal A(Z_n).
\end{equation}
In fact, from \eqref{hol}, \eqref{mono} and \eqref{qnapp} we obtain that
\[
\mathcal B^{n+1} = \mathcal{QN}_p(Z_n)^{n+1} \subset   \mathcal{QN}_{p/(n+1)}(Z_n)  \subset \mathcal{QN}_2(Z_n)  \subset \mathcal A(Z_n)
\]
since $(n+1)/p \le 1$ and $p/(n+1) \le 2$.  This implies  \eqref{qnp1}, since
the operator norm satisfies $\Vert \cdot \Vert \le \vert \cdot \vert_p$ by \eqref{mono}.

By combining these facts we reach  the contradiction that $S^n \in \mathcal A(X_n)$. Thus $\mathcal J_{n} \notin \mathcal{SI}_{n-1}(\mathcal B)$. 
 \end{proof}
 
 \smallskip

Theorem \ref{general} also applies to a class of  $p$-compact operators, which is related to $\mathcal{QN}_p$ by duality.
Let $X$ and $Y$ be Banach spaces, and $1 < p < \infty$. We say that  $T \in \mathcal L(X,Y)$ is a \textit{(Sinha-Karn) $p$-compact} operator, 
denoted $T \in \mathcal{SK}_p(X,Y)$, if there is a strongly $p$-summable sequence $(y_k)$ in $Y$ such that 
\[
T(B_X) \subset \Big\{\sum_{k=1}^\infty a_ky_k: (a_k) \in B_{\ell^{p'}}\Big\}.
\] 
Here  $1/p + 1/p' = 1$ and $B_Z$ is the closed unit ball of the space $Z$. 
There is a complete norm $\vert \cdot \vert_{ \mathcal{SK}_p}$ such that  $( \mathcal{SK}_p, \vert \cdot \vert_{ \mathcal{SK}_p})$ 
defines a Banach operator ideal, see e.g. \cite[Theorem 4.2]{SK02}. 
(The classes $\mathcal{SK}_p$ were introduced in \cite{SK02} as the $p$-compact operators, but we use the above terminology to distinguish them from 
an earlier  notion of  $p$-compactness in the literature, see e.g. \cite[p. 529]{Pi14}.)
Properties of $\mathcal{SK}_p$ have been studied quite intensively, see e.g.  \cite{O12}, \cite{GLT12},  
\cite[Section 3]{W23} and their  references. 
By \cite[Theorem 2.8]{GLT12} one has 

$\bullet$  $T \in \mathcal{SK}_p(X,Y)$ if and only if $T^* \in \mathcal{QN}_p(Y^*,X^*)$, 

$\bullet$  $T \in \mathcal{QN}_p(X,Y)$ if and only if $T^* \in \mathcal{SK}_p(Y^*,X^*)$,

\noindent where the respective ideal norms are preserved. Hence 
the classes $\mathcal{SK}_p$  satisfy the exact analogues of \eqref{mono}, \eqref{hol} and \eqref{qnapp}.

We formulate a variant for $\mathcal{SK}_p$ of part (i) of Theorem \ref{nsubid}, which does not directly follow by 
duality from the analogous result for $\mathcal{QN}_p$.

 \begin{example}\label{nsubidsk} 
  Let $n \in \mathbb N$ and $p \in (2n, 2(n+1)]$. 
 Then there is a closed subspace $Z_n \subset c_0$ and a closed subalgebra  $\mathcal M_{n}$ of $\overline{\mathcal{SK}_p(Z_n)}$ such that 
  \[
  \mathcal M_{n} \in \mathcal{SI}_{n+1}(\mathcal L(Z_n)) \setminus \mathcal{SI}_{n}(\mathcal L(Z_n)).
  \] 
 \end{example}
 
 \begin{proof}
 The argument  follows the general outline  of  Theorem \ref{nsubid}.(i).
 According to \cite[Proposition 4.3]{W23} there is a closed subspace $Y_n\subset c_0$ as well as 
$V \in\mathcal{SK}_p(Y_n)$ such that $V^n\notin \mathcal A(Y_n)$. 

We will also require the fact that  if $2n <  p \le 2(n+1)$, then
\begin{equation}\label{powers2} 
\overline{\mathcal{SK}_p(Y)^{n+1}} = \mathcal A(Y)
\end{equation}
holds for any Banach space $Y$.
This equality is verified exactly as in the proof of  \eqref{powers1} by using the $\mathcal{SK}_p$-versions of properties \eqref{mono} - \eqref{qnapp}
(cf.  \cite[Remarks 4.2.(ii)]{W23}).

Let $Z_n = Y_n \oplus Y_n$. The result follows by applying Theorem \ref{general} for the closed ideal $\mathcal J = \overline{\mathcal{SK}_p(Z_n)}$
of $\mathcal L(Z_n)$. The desired $(n+1)$-subideal has the form
\begin{equation*}
\mathcal M_{n}= span(T,\ldots,T^{n})+ \overline{\mathcal{SK}_p(Z_n)^{n+1}},
\end{equation*}
where $T=
 \begin{bmatrix}
V&0\\
0&0
\end{bmatrix}\in \mathcal{SK}_p(Z_n)$.
 \end{proof}
     
 
 \section{Concluding remarks}\label{conclude}

We include here a  mixed collection  of results that clarify various  aspects of  Section  \ref{nsubs}. For instance, 
in  the context of Theorem \ref{general}  there is additional information on certain powers of the closed subideals constructed therein.
We formulate a sufficient condition that yields closed $n$-subideals of general  Banach algebras,
and apply it in Theorem \ref{tarbard2} to find higher closed subideals related to the Tarbard spaces.

Let $\mathcal  A$ be a Banach algebra. 
The ideal of  $\mathcal  A$ generated by the subset $M\subset \mathcal  A$ is denoted by
\[
[M]_{\mathcal A} :=\bigcap\{\mathcal I  : \ \mathcal I \textrm{  is an ideal of } \mathcal  A \textrm{ containing } M\}, 
\]
and the closure
\[
\overline{[M]}_{\mathcal A} = \bigcap\{\mathcal I  : \ \mathcal I \textrm{  is a closed ideal of } \mathcal  A \textrm{ containing } M\}
\]
is the corresponding closed ideal of $\mathcal A$ generated by $M$.  It is well known that
\begin{equation}\label{generatedideal}
[M]_{\mathcal A}= \mathcal A M+M \mathcal A+ \mathcal A M  \mathcal A+M,
\end{equation}
where $\mathcal A M=span\{am:  a\in \mathcal A, m\in M\}$, and $M \mathcal A$ as well as  $\mathcal A  M  \mathcal A$ are defined similarly.
We abbreviate our notation as $[M] := [M]_{\mathcal L(X)}$ and $\overline{[M]} := \overline{[M]}_{\mathcal L(X)}$
for $\mathcal  A  = \mathcal L(X)$, where $X$ is a Banach space and $M \subset \mathcal L(X)$.

We first discuss powers of closed  $n$-subideals  for Banach algebras $\mathcal  A$. By way of  motivation we recall  a lemma of Andrunakievich  
concerning (ring theoretic) subideals  from ring theory, where the setting need not allow for scalar multiplication.

Let $\mathcal R$ be a ring, and suppose that  $\mathcal J \subset \mathcal I \subset \mathcal R$ are subrings. We say here that 
$\mathcal J$ is a (ring) \textit{subideal} of $\mathcal R$ if $\mathcal J$ is a ring  ideal of $\mathcal I$ and  $\mathcal I$ is a ring ideal of $\mathcal R$. 
Such subideals occur in the following  observation  of Andrunakievich (restated in this terminology):
\textit{if $\mathcal J$ is a subideal of the ring $\mathcal R$, then}
\begin{equation}\label{al}
(\mathcal J)^3_{\mathcal R} \subset \mathcal J,
\end{equation}
see e.g. \cite[Lemma 4]{A58} or \cite[p. 11]{GW04}. 
Above $(\mathcal J)_{\mathcal R}$ is the ring ideal of $\mathcal R$ generated by $\mathcal J$, and the power $(\mathcal J)^3_{\mathcal R}$ 
is the ring ideal of $\mathcal R$ generated by the set 
$\{a_1a_2a_3:  a_j \in (\mathcal J)_{\mathcal R} \textrm{ for } j = 1, 2, 3\}$.

The following result is a version of \eqref{al} for Banach algebras. It is known that 
$\mathcal K(\ell^2)$  contains ring ideals that are not algebra ideals, see  \cite[Example 4.1]{PW14}, so  Proposition \ref{andru} concerns  
a different type of ideals compared to \eqref{al}. 

\begin{proposition}\label{andru}
If $\mathcal A$  is a Banach algebra and  $\mathcal J\in\mathcal{SI}_{2}(\mathcal A)$, then 
 \[
 \overline{[\mathcal J]_{\mathcal A}^3} \subset \mathcal J.
 \]
\end{proposition}

\begin{proof}
By linearity and continuity it is enough to verify that $s_1s_2s_3\in \mathcal J$ for any $s_1,s_2,s_3\in [\mathcal J]_A$, since
$\mathcal J$ is a closed subalgebra of $\mathcal A$. Recall from \eqref{generatedideal} that
\[
[\mathcal J]_{\mathcal A} = \mathcal  A \mathcal  J + \mathcal   J  \mathcal   A+ \mathcal   A  \mathcal J  \mathcal   A+  \mathcal  J.
\] 
We may again assume by linearity that $s_2=u_1t_1+t_2u_2+u_3t_3u_4+t_4$ where $u_i\in  \mathcal   A$ and $t_i\in  \mathcal  J$ for  $i=1,\ldots,4$. Thus
\begin{equation}\label{1908}
s_1s_2s_3=s_1u_1t_1s_3+s_1t_2u_2s_3+s_1u_3t_3u_4s_3+s_1t_4s_3.
\end{equation}
Here $s_1u_1\in [\mathcal J]_{\mathcal A}$ as $[\mathcal J]_{\mathcal A}$ is an ideal of $ \mathcal A$. By assumption $\mathcal J \subset \mathcal I \subset \mathcal A$,
where $\mathcal I$ is a closed ideal of $\mathcal A$ and $\mathcal J$ is an ideal of $\mathcal I$. Thus $[\mathcal J]_{\mathcal A} \subset \mathcal I$, so that 
$\mathcal  J$ is an ideal of $[\mathcal J]_{\mathcal A}$. This implies that  $t_1s_3\in  \mathcal  J$, and thus  $s_1u_1t_1s_3\in  \mathcal  J$.

In a similar manner one verifies that the other three terms on the right hand side of \eqref{1908} belong to $\mathcal   J$, whence  $s_1s_2s_3\in  \mathcal  J$.
\end{proof}  

 Concerning the algebras $\mathcal L(X)$ the following extension of  Theorem \ref{general}  provides further information for e.g.  the closed $(n+1)$-subideals obtained in  Theorems \ref{nsubids2}, \ref{nsubids3} and \ref{nsubid}
 by applying Theorem \ref{general}.

\begin{theorem}\label{general+}
Let $n\in\mathbb N$. Suppose that $X$ is a Banach space for which there is a closed ideal  $\mathcal I\subset \mathcal L(X)$ and an operator $S\in \mathcal I$ 
such that 
\begin{equation}\label{crit+}
S^n\notin \overline{\mathcal I^{n+1}}.
\end{equation} 
Put $Z=X\oplus X$ and let 
\[
\mathcal M_n \in \mathcal{SI}_{n+1}(\mathcal L(Z)) \setminus \mathcal{SI}_{n}(\mathcal L(Z))
\]
be the closed $(n+1)$-subideal defined in \eqref{nsubgen}.  Then  
\begin{itemize}
   \item[(i)]\    $\overline{[\mathcal M_n]^{n+1}}\varsubsetneq \mathcal M_n$. 
    
  \item[(ii)]\   $[\mathcal M_n]^n \not\subset \mathcal M_n$. 
  \end{itemize}
\end{theorem}

\begin{proof}
(i) Recall from \eqref{nsubgen} and  the proof of Theorem \ref{general} that 
\[
\mathcal M_n = span(T, T^2,\ldots,T^n) + \overline{\mathcal J^{n+1}},
\]
where 
\[
 \mathcal J := \begin{bmatrix}
\mathcal I&\mathcal I\\\mathcal I&\mathcal I
\end{bmatrix} \subset \mathcal L(Z)  \textrm{ and }  
T:=\begin{bmatrix}
S&0\\0&0
\end{bmatrix}\in \mathcal J.
\]
Since $\mathcal J$ is a closed ideal of $\mathcal L(Z)$ and $\mathcal M_n \subset \mathcal J$ by definition, we get that
$[\mathcal M_n] \subset \mathcal J$ and  $[\mathcal M_n]^{n+1}\subset \mathcal J^{n+1}$.
After  passing to the closures it follows that 
\[
\overline{[\mathcal M_n]^{n+1}}\subset \overline{\mathcal J^{n+1}} \subset \mathcal M_n.
\]
Here $\overline{[\mathcal M_n]^{n+1}}\varsubsetneq \mathcal M_n$,  since $\overline{[\mathcal M_n]^{n+1}}$ is a closed ideal of $\mathcal L(Z)$, while
$\mathcal M_n$  fails to be a closed ideal of $\mathcal L(Z)$ by construction.

\smallskip

(ii) For $n = 1$ it suffices to note that  $\mathcal M_1$ is not an ideal of $\mathcal L(Z)$, so that  $\mathcal M_1 \varsubsetneq [\mathcal M_1]$. For $n \ge 2$ 
recall that 
\[
\begin{bmatrix}
0&0\\
S^n&0
\end{bmatrix} =
\left(\begin{bmatrix}
0&0\\
I_X&0
\end{bmatrix}\circ T\right)\circ T^{n-1}
\in  [\mathcal M_n]^n,
\]
since $\begin{bmatrix}
0&0\\
I_X&0
\end{bmatrix}\circ T \in  [\mathcal M_n]$. 
It was shown during  the proof of Theorem \ref{general} (see \eqref{1409})
that 
$\begin{bmatrix}
0&0\\
S^n&0
\end{bmatrix} \notin \mathcal M_n$.
Thus $[\mathcal M_n]^n \not\subset \mathcal M_n$ for $n \ge 2$. 
\end{proof}

In view of Theorem \ref{general+}.(i)  there are Banach spaces $X$ and non-trivial closed $4$-subideals $\mathcal M \in \mathcal{SI}_4(\mathcal L(X)) \setminus \mathcal{SI}_3(\mathcal L(X))$
for which  $\overline{[\mathcal M]^4} \varsubsetneq  \mathcal M$ and $[\mathcal M]^3 \not\subset \mathcal M$. 
A comparison with Proposition  \ref{andru}  suggests  the following 
 
\medskip

\noindent \textbf{Problem}.  Does there exist a space $X$ and $\mathcal M \in \mathcal{SI}_3(\mathcal L(X))$ for which 
\begin{equation}\label{and}
[\mathcal M]^3 \not\subset \mathcal M \ ?
\end{equation}
More generally, is there a Banach algebra $\mathcal A$ and $\mathcal M \in \mathcal{SI}_3(\mathcal A)$ 
such that \eqref{and} holds?

\smallskip

In Section \ref{nsubs} we have studied non-trivial closed  higher subideals of $\mathcal L(X)$. 
It is a pertinent  query for which other basic classes of Banach algebras
\[
\mathcal{SI}_{n}(\mathcal A) \varsubsetneq \mathcal{SI}_{n+1}(\mathcal A)
\]
for some $n \in \mathbb N$. We did not pursue this here, but  we 
include a nilpotency criterion for a subalgebra of a  non-unital Banach algebra to be an $n$-subideal. 
Condition  \eqref{t} only applies to $\mathcal  M = \mathcal A$ if $\mathcal  A$ has a unit.

 \begin{prop}\label{lemma}
Let $n\ge 2$ and suppose that
 $\mathcal A$ is a non-unital Banach algebra. If $\mathcal  M \varsubsetneq \mathcal A$ is a subalgebra  such that
\begin{equation}\label{t}
\mathcal A^{n+1}\subset \mathcal M,
\end{equation}
then $\mathcal  M$ is an $n$-subideal of $\mathcal A$. If  $\mathcal  M$ is also closed in $\mathcal A$, then 
$\mathcal  M \in \mathcal{SI}_{n}(\mathcal A)$.
\end{prop}

\begin{proof}
Let $\mathcal I_0:= \mathcal A$ and define successively 
\[
\mathcal I_k=[\mathcal M\cap \mathcal I_{k-1}]_{\mathcal I_{k-1}} \quad \textrm{ for } k=1,\ldots,n-1. 
\]
One verifies  by induction that $\mathcal  M\subset \mathcal I_k$ for $k=0,\ldots, n-1$, so that actually $\mathcal I_k=[\mathcal M]_{\mathcal I_{k-1}}$ for $k=1,\ldots, n-1$. Hence,
\[
\mathcal  M \subset \mathcal I_{n-1}\subset\cdots\subset \mathcal I_2\subset \mathcal I_1\subset \mathcal I_0= \mathcal A.
\]
Here $\mathcal I_{k}$ is an ideal of $\mathcal I_{k-1}$  by definition for $k=1,\ldots, n-1$, so it remains to verify that $\mathcal  M$ is an ideal of $\mathcal I_{n-1}$.
Towards this, we first check  the following auxiliary result by induction.

\smallskip

\textit{Claim.} 
\begin{equation}\label{claim}
\mathcal I_{k}\subset \mathcal A^{k+1}+\mathcal  M \quad  \textrm{ for all } k=0,\ldots, n-1.
\end{equation} 

The case $k=0$  is clear since $\mathcal I_0= \mathcal  A$. The induction assumption is that \eqref{claim} holds for some $k\in\{0,\ldots, n-2\}$. 
Since 
\[
\mathcal I_{k+1}=[\mathcal  M]_{\mathcal I_k}= \mathcal I_k \mathcal M+\mathcal  M  \mathcal I_k+ \mathcal I_k \mathcal  M  \mathcal I_k+\mathcal M
\]
according to \eqref{generatedideal}, we will  verify in (i) - (iii) below that the products
\[
a_1m,\   ma_2,\  a_1ma_2  \in \mathcal A^{k+2}+\mathcal M
\]
whenever  $a_1,a_2\in \mathcal I_k$ and  $m\in \mathcal M$.

\smallskip

(i)  The induction hypothesis  $\mathcal I_{k}\subset \mathcal A^{k+1}+\mathcal  M$ yields that  
$a_1=a+b$, where $a\in \mathcal A^{k+1}$ and $b\in \mathcal M$. Since $m\in \mathcal M \subset  \mathcal A$ we get that  
$am\in \mathcal A^{k+2}$. 
Moreover, $bm\in \mathcal M$ since $\mathcal M$ is a subalgebra of $\mathcal A$. Consequently  
\begin{equation}\label{eq15}
a_1m=am+bm\in \mathcal A^{k+2}+ \mathcal M.
\end{equation}

(ii) A symmetrical argument to (i) ensures that  $ma_2\in \mathcal A^{k+2}+\mathcal M$.
\smallskip

(iii)  From  \eqref{eq15} we get that  $a_1m=c+d$, where $c\in \mathcal A^{k+2}$ and $d\in \mathcal M$.
Moreover, $a_2\in \mathcal  A^{k+1}+ \mathcal M$ by the induction hypothesis, so that 
$a_2=c_1+c_2$ with  $c_1\in \mathcal  A^{k+1}$ and $c_2\in \mathcal M$.  It follows that  $dc_1\in \mathcal  A^{k+2}$ and $dc_2\in \mathcal M$. Thus 
$da_2=dc_1+dc_2\in \mathcal  A^{k+2}+ \mathcal M$,
so that 
\[
a_1ma_2=(c+d)a_2=ca_2+da_2\in \mathcal  A^{k+2}+\mathcal M
\]
as $ca_2\in \mathcal  A^{k+3}\subset \mathcal  A^{k+2}$. Thus $\mathcal I_{k+1}\subset \mathcal A^{k+2}+\mathcal  M$ by linearity which completes the proof of the Claim \eqref{claim}.

\smallskip

To conclude, suppose that  $m\in \mathcal M$ and $a\in \mathcal  I_{n-1}$ are arbitrary. 
From \eqref{claim} with $k = n-1$  we get that  $a=b+c$, where $b\in \mathcal  A^n$ and $c\in \mathcal M$. 
It follows that $bm\in \mathcal  A^{n+1} \subset \mathcal M$ by the assumption \eqref{t}. Moreover, $cm\in \mathcal M$ as $\mathcal M$ is a subalgebra of $\mathcal  A$, 
so that $am=bm+cm\in \mathcal M$. A symmetrical argument yields that $ma\in \mathcal M$. Consequently   $\mathcal M$ is an ideal of $\mathcal  I_{n-1}$, 
so that  $\mathcal M$ is an $n$-subideal of $\mathcal A$.

Finally, suppose that the closed subalgebra  $\mathcal M$  satisfies  \eqref{t}. The first part of the argument gives  
a chain $\mathcal M=\mathcal  J_n\subset \mathcal  J_{n-1}\subset\cdots\subset \mathcal  J_1\subset \mathcal  J_0= \mathcal A$ of relative ideals.
It follows that   
\[
\mathcal M=\overline{\mathcal  J_n}\subset\overline{\mathcal  J_{n-1}}\subset\cdots\subset \overline{\mathcal  J_1}\subset \overline{\mathcal  J_0}= \mathcal  A
\] 
is a chain of closed relative ideals associated to  $\mathcal M$, 
so that $\mathcal M \in \mathcal{SI}_n(\mathcal A)$.
\end{proof}

The nilpotency criterion \eqref{t} can be used in some arguments  in Section  \ref{nsubs}.
However, we prefer to define the relevant higher subideals as explicitly as possible,  whereas  
the chain of closed relative ideals found in the proof of Proposition \ref{lemma} is rather implicit. We next apply Proposition \ref{lemma} to the closed $n$-subideals in the setting of  the  Tarbard 
spaces from  Section \ref{Tarbard}, where the fact (ii) formulated in the proof is quite unexpected.

 \begin{theorem}\label{tarbard2}
Let $Y$ be the Tarbard space $X_{2n}$ or $X_{2n+1}$, where $n\ge 2$. Then
\[
\mathcal{SI}_1(\mathcal L(Y))\subsetneq \cdots\subsetneq \mathcal{SI}_{n-1}(\mathcal L(Y))\subsetneq \mathcal{SI}_{n}(\mathcal L(Y))=\mathcal{SI}_{n+1}(\mathcal L(Y))=\cdots
\]  
\end{theorem}
  
\begin{proof}
According to part (1) of Remarks \ref{sect4}  it is enough to verify that
\begin{itemize}
\item[(i)] $\mathcal{SI}_{n-1}(\mathcal L(Y))\subsetneq \mathcal{SI}_n(\mathcal L(Y))$ and
\item[(ii)] $\mathcal{SI}_{n}(\mathcal L(Y))=\mathcal{SI}_{n+1}(\mathcal L(Y))$.
\end{itemize}

Let $S \in \mathcal S(Y)$ be the operator constructed by Tarbard \cite{Ta12},  for which $S^{2n-1} \notin \mathcal K(Y)$ and $S^{2n+1} = 0$.
(For $Y=X_{2n}$ we have $S^{2n}=0$, but this fact will play no role here.)

Towards (i) we define 
\[
\mathcal I= span(S^{2},S^4,\ldots,S^{2n}) + \mathcal K(Y),
\] 
which is a closed subalgebra of $[S^2]$. 

Recall from \eqref{sj} that $[S^j]=span(S^j,\ldots,S^{k-1})+\mathcal K(Y)$ for $j=1,\ldots,k-1$. Hence it follows by iteration from \eqref{prod2}  that 
\begin{equation}\label{7925}
[S^2]^n\subset  span(S^{2n}) +  \mathcal K(Y) \subset \mathcal I.
\end{equation} 
Thus $\mathcal I\in \mathcal{SI}_{n-1}([S^2])$ by Proposition \ref{lemma}, and 
$\mathcal I\in \mathcal{SI}_n(\mathcal L(Y))$.

It remains to observe that $\mathcal I\notin \mathcal{SI}_{n-1}(\mathcal L(Y))$. However, this follows directly from Lemma \ref{newlemma} since $S \cdot (S^2)^{n-1}=S^{2n-1} \notin\mathcal I$ by \eqref{bded}.
\smallskip

(ii) The inclusion $\mathcal{SI}_{n}(\mathcal L(Y)) \subset \mathcal{SI}_{n+1}(\mathcal L(Y))$ follows from the definition. Conversely, suppose that  $\mathcal J\in \mathcal{SI}_{n+1}(\mathcal L(Y))$ and let
\[
\mathcal J=\mathcal J_{n+1}\subset\cdots\subset\mathcal J_1\subset\mathcal L(Y)
\]
be an associated  chain of closed relative ideals. We first recall that $\mathcal K(Y)\subset \mathcal J$ by Proposition \ref{propA} as $Y$ has the A.P.

If $\mathcal J_1=[S]$, then $\mathcal J_2$ is a closed ideal of $\mathcal L(Y)$ by Lemma \ref{codim}, so that $\mathcal J\in \mathcal{SI}_n(\mathcal L(Y))$. 
Thus we may assume that $\mathcal J_1\neq [S]$, whence  $\mathcal J_1\subset[S^2]$ by \eqref{idchain}.

Observe next that if $\mathcal J\subset[S^3]$, then we deduce from \eqref{prod2}  that  
\[
[S^3]^n\subset \mathcal K(Y)\subset\mathcal J
\]
as $3n\ge 2n+1$. 
Thus $\mathcal J\in \mathcal{SI}_{n-1}([S^3])$ by Proposition \ref{lemma}, and  $\mathcal J\in \mathcal{SI}_n(\mathcal L(Y))$.

On the other hand, if $\mathcal J\not\subset[S^3]$ then there is an operator 
\[
T=\sum_{j=2}^{2n}a_jS^j+K_0\in \mathcal J,
\] 
where $a_2\neq 0$ and $K_0$ is compact, since $\mathcal J\subset\mathcal J_1\subset[S^2]$. By a similar calculation as in the proof of \eqref{prod2} we have
\[
T^n=a_2^n S^{2n}+K\in \mathcal J
\]
for some $K\in \mathcal K(Y)$, and thus $S^{2n}\in \mathcal J$ by linearity. It follows that
\[
[S^2]^n\subset span(S^{2n}) + \mathcal K(Y) \subset \mathcal J
\] 
where the first inclusion was observed in \eqref{7925}.

 Finally, Proposition \ref{lemma} yields that $\mathcal J\in \mathcal{SI}_{n-1}([S^2])$, so $\mathcal J\in \mathcal{SI}_n(\mathcal L(Y))$.
 This concludes the argument.
\end{proof}

\begin{remark}\label{tarbardtoo}
Let  $X_n$ be the Tarbard space for $n \ge 2$. We note that Theorem \ref{general} yields  
a closed subalgebra $\mathcal M\subset \mathcal S(Z_n)$ such that 
\[
\mathcal M \in \mathcal{SI}_{n+1}(\mathcal L(Z_n)) \setminus \mathcal{SI}_{n}(\mathcal L(Z_n)),
\]
where $Z_n := X_{n+1} \oplus X_{n+1}$ for $n \in \mathbb N$.  Towards  this one requires 
the above operator  $S \in \mathcal S(X_{n+1})$,  for which $S^{n} \notin \mathcal K(X_{n+1})$ and $S^{n+1} = 0$, as well as 
\begin{equation}\label{niltarbard}
\overline{\mathcal S(X_{n+1})^{n+1}} = \mathcal K(X_{n+1}).
\end{equation}
In fact,  $V_1 \cdots V_{n+1} \in \mathcal K(X_{n+1})$ for any $V_1,\ldots,V_{n+1} \in \mathcal S(X_{n+1})$ by an iteration of  \eqref{prod2}, 
so condition \eqref{niltarbard} is valid as $X_{n+1}$ has the A.P.

Above  $Z_n$ is not an H.I. space, whereas  in Theorem \ref{tarbard2} the spaces $X_k$ are  H.I. for all $k$.
\end{remark} 

Theorem \ref{nsubidchain} suggests the following problem.

\smallskip

\noindent \textbf{Problem}. Let  $X_n = \bigoplus_{j=1}^{n+1} \ell^{p_j}$ be the direct sum space from \eqref{elem}. 
Is there a decreasing sequence $(\mathcal M_k)$ of closed 
subalgebras of $\mathcal L(X_n)$ such that 
 \[
\mathcal M_k \in \mathcal{SI}_{k+1}(\mathcal L(X_n)) \setminus \mathcal{SI}_{k}(\mathcal L(X_n)), \quad  k \in \mathbb N?
 \]
If such a sequence $(\mathcal M_k)$ exists, then Proposition \ref{lemma} implies that  each quotient algebra $\mathcal M_k/\mathcal K(X_n)$ must be non-nilpotent, which lies outside of the setting of the $\mathcal S(X_n)$-subideals considered in Theorem \ref{nsubids2}. However, 
there are closed ideals $\mathcal I \subset \mathcal L(X_n)$ for which $\mathcal I/\mathcal K(X_n)$ is not nilpotent. (In fact, 
$\mathcal L(X_n)$ has precisely $n+1$ maximal closed ideals which are non-nilpotent modulo $\mathcal K(X_n)$, see \cite{Vo76}.)

\smallskip
  
Finally, for completeness we display a few  examples of closed subalgebras of  $\mathcal L(X)$ that fail to be
$n$-subideals for any $n \in \mathbb N$. Recall from Proposition \ref{propA} that $\mathcal A(X) \subset \mathcal M$ for any 
non-zero closed $n$-subideal $\mathcal M \subset \mathcal L(X)$. Let  $X$ be any infinite-dimensional Banach space and put
\[
\mathcal M = span(x^* \otimes x)  \subset \mathcal L(X),
\]
where $x \in X$ and $x^* \in X^*$ are normalised vectors. It follows that  the closed subalgebra $\mathcal M \notin \mathcal{SI}_{n}(\mathcal L(X))$ for any $n \in \mathbb N$.
For another large  class of such examples,  let
 \[
 \mathcal M = span(I_X) + \mathcal B,
 \]
where  $\mathcal B \subset \mathcal L(X)$ is any closed subalgebra such that the quotient space $\mathcal L(X)/\mathcal B$ is at least $2$-dimensional. 
In this case  one applies  the fact that if $\mathcal I$ is an $n$-subideal of $\mathcal L(X)$ for some $n \in \mathbb N$ and  $I_X \in \mathcal I$, then
$\mathcal I = \mathcal L(X)$ by Lemma \ref{newlemma}.
In particular, if $\mathcal L(X)/\mathcal K(X)$ is at least $2$-dimensional, then  $\mathcal M = span(I_X) + \mathcal K(X)$
is a closed Lie ideal of $\mathcal L(X)$ which is not an $n$-subideal for any $n \in \mathbb N$.

\smallskip
 
There  are subtler examples of this phenomenon, including closed subalgebras $\mathcal M$ that satisfy $\mathcal A(X) \varsubsetneq \mathcal M \varsubsetneq \mathcal K(X)$
for certain spaces $X$.

 \begin{example}\label{nonil2} 
There is a Banach space $Z$ without the A.P. and a closed subalgebra $\mathcal A(Z) \subset  \mathcal M \subset \mathcal K(Z)$,
 such that $\mathcal M \notin \mathcal{SI}_{n}(\mathcal L(Z))$ for any $n \in \mathbb N$.
 \end{example}
 
 \begin{proof} 
 Let $X$ be a Banach space without the A.P. for which there is an operator $S\in \mathcal K(X)$ such that $S^n\notin\mathcal A(X)$ for any $n\in\mathbb N$. Recall from Remark \ref{WST} that such spaces $X$ exist. 
 
 Put $Z=X\oplus X$ and $T=\begin{bmatrix}
S&0\\
0&0
\end{bmatrix} \in \mathcal K(Z)$. Let
\[
M_0: = span(\{T^n:  n\in\mathbb N\})+\mathcal A(Z).
\]
It is straightforward to verify that $\mathcal M:=\overline{M_0}$ is a closed subalgebra of $\mathcal L(Z)$, for which 
$\mathcal A(Z)\subset  \mathcal M\subset\mathcal K(Z)$. Suppose to the contrary that $\mathcal M \in \mathcal{SI}_{n}(\mathcal L(Z))$ for some $n \in \mathbb N$. Then
\[
\begin{bmatrix}
0&0\\
S^{n+1}&0
\end{bmatrix} = \begin{bmatrix}
0&0\\
S&0
\end{bmatrix}\circ T^n  \in  \mathcal M
\] 
by Lemma \ref{newlemma}.
On the other hand,  if $U = \begin{bmatrix}
U_{11} &U_{12}\\
U_{21}&U_{22}
\end{bmatrix} \in  \mathcal M$,
then by approximation the component $U_{21} \in \mathcal A(Z)$. This entails  that  
$S^{n+1}\in\mathcal A(X)$, which  contradicts the choice of  $S$.
 \end{proof}

\smallskip

\textit{Acknowledgements}. We are indebted to Niels Laustsen (Lancaster)  for a question at the conference  
\textit{Structures in Banach spaces} (Erwin Schr\"odin\-ger Institute, Vienna) in March 2025
which led to Theorem \ref{nsubidchain}, and to Tomasz Cia\'s (Pozna\'n) for a discussion which motivated Theorem \ref{newsubids}.

\bibliographystyle{amsplain}

\bibliography{bibliography}

\end{document}